\documentclass{amsart}
\usepackage[utf8]{inputenc}

\title[Abelian instances of nonabelian symplectic reduction]{Abelian instances of nonabelian symplectic reduction}  

\author[A.\ Bravo-Doddoli]{Alejandro\ Bravo-Doddoli} 
\address{Alejandro Bravo: Department of Mathematics, University of Michigan, Ann Arbor, MI 48109, U.S.}
\email{Abravodo@umich.edu}
\author[L.C.\ Garc\'ia-Naranjo]{Luis\ C.\ Garc\'ia-Naranjo} 
\address{Dipartimento di Matematica “Tullio Levi-Civita”, Universit\`a di Padova,
Via Trieste 63, 35121 Padova, Italy}
\email{luis.garcianaranjo@math.unipd.it}
\author[E.\ Rigato]{Enzo\ Rigato} 
\address{Dipartimento di Matematica “Tullio Levi-Civita”, Universit\`a di Padova,
Via Trieste 63, 35121 Padova, Italy}
\email{enzo.rigato@studenti.unipd.it}
\keywords{symplectic reduction, coadjoint orbits, metabelian groups, Carnot groups}

\date{May 2026}

\usepackage[english]{babel}
\usepackage[utf8]{inputenc}
\usepackage{libertine}
\usepackage{graphicx}
\usepackage{floatflt}
\usepackage{blindtext}
\usepackage{enumitem}
\usepackage{amsthm}
\usepackage{subfig}
\usepackage{listings}
\usepackage{listingsutf8}
\usepackage{amsmath}
\usepackage{framed}
\usepackage{minibox}
\usepackage{float}
\usepackage{wrapfig}
\usepackage{longtable}
\usepackage[strict]{changepage}
\usepackage{pgfplots}
\usepackage{nicefrac}
\usepackage{units}
\usepackage{etoolbox,tikz}
\usepackage{bbm}
\usepackage{hyperref}
\usepackage[autostyle]{csquotes}

\usetikzlibrary{external}
\usetikzlibrary{cd}

\AtBeginEnvironment{tikzcd}{\tikzexternaldisable}
\AtEndEnvironment{tikzcd}{\tikzexternalenable}

\usepackage{amsmath}
\usepackage{amsfonts}

\usepackage{amsthm}
\usepackage {amsmath, amssymb}
\usepackage{pb-diagram}
\usepackage{tikz-cd}
\usepackage{mathtools}
\usepackage{amssymb}

\usetikzlibrary{matrix}
\pgfplotsset{width=11cm,compat=1.9}
\usepgfplotslibrary{external}
\DeclareMathOperator{\B}{B}

\DeclareMathOperator{\ad}{ad}
\DeclareMathOperator{\Ad}{Ad}
\DeclareMathOperator{\SE}{SE}
\DeclareMathOperator{\SO}{SO}

\DeclareMathOperator{\T}{T}
\DeclareMathOperator{\Z}{Z}
\DeclareMathOperator{\C}{C}

\DeclareMathOperator{\poj}{p}
\DeclareMathOperator{\V}{V}

\def\R{\mathbb{R}}

\def\J{\mathcal{J}}

\usepackage{hyperref}
\usepackage[%
	capitalize,nameinlink
]{cleveref}

\hypersetup{%
	colorlinks=true,
	linkcolor=blue,
	citecolor=magenta,
	urlcolor=magenta,
}

\numberwithin{figure}{section}

\def\ma{metabelian }

\newtheorem{theorem}{Theorem}

\newtheorem{thm}{Theorem}
\newtheorem{proposition}[thm]{Proposition}

\newtheorem{lemma}[thm]{Lemma}
\newtheorem{definition}[thm]{Definition}

\theoremstyle{definition}
\newtheorem{remark}{Remark}

\newcommand{\Laz}{\mathfrak{z}}
\newcommand{\Lag}{\mathfrak{g}}
\newcommand{\Laa}{\mathfrak{a}}

\newcommand{\Lah}{\mathfrak{h}}
\newcommand{\Ag}{\mathbb{A}}
\newcommand{\G}{\mathbb{G}}
\newcommand{\Hg}{\mathbb{H}}

\newcommand{\Lagr}{\mathfrak{g}_{\mathrm{\tiny reg}}^*}
\newcommand{\dslash}{/\!\!/}

\begin{document}

\begin{abstract}

 Let $\G$ be a Lie group with a normal abelian subgroup $\Ag$, and let $(M,\omega)$ be a symplectic manifold endowed with a Hamiltonian $\G$-action. We investigate conditions under which symplectic reduction by $\G$ coincides with the symplectic reduction by the abelian subgroup $\Ag$. 
Using the reduction-by-stages framework (Marsden et al Springer Notes in Math., 1913, (2007)\cite{stages}), we prove that, under a mild assumption, the corresponding reduced spaces are symplectomorphic if and only if they have the same dimension. Both this assumption and the dimension condition depend only on the groups $\G$ and $\Ag$, and on the momentum value
$\mu\in \Lag^*$ at which the symplectic reduction by $\G$ is performed; in particular, they are independent of the symplectic manifold $(M,\omega)$. 
We then provide a broad class
of examples by identifying a 
large family of nilpotent Lie groups, including classical Carnot groups such as the Heisenberg group and jet-space $\J^k(\R^n,\R^m)$, for which the two reduced spaces are symplectomorphic for  generic momentum values.

\end{abstract}


\maketitle

\section{Introduction}

Let  $\G$ be a Lie group with  
a regular, normal, abelian subgroup $\Ag$. 
Suppose that $\G$ acts in a Hamiltonian fashion
on a symplectic manifold $(M,\omega)$. This
action can be restricted to a Hamiltonian
action of the abelian subgroup $\Ag$.
In this paper, we investigate conditions that lead to the equivalence
of the symplectic reduction by the (generally nonabelian) full Lie group $\G$ and the 
abelian subgroup $\Ag$. Our inspiration was the observation that this equivalence holds for generic momentum values for a large number of low-dimensional nilpotent groups (see Table \ref{Table:examples} below).

\subsection{Main contributions} Denote by $\Lag$
and $\Laa$ the  Lie algebras of $\G$ and $\Ag$, 
and by
$i:\Laa \hookrightarrow \Lag$ the natural inclusion map. Let 
$i^*:\Lag^*\to \Laa^*$ be the corresponding dual map (which
is the restriction of elements of $\Lag^*$ to $\Laa$). We prove and present a wide class of examples of the following.

\begin{theorem}
\label{th:main-new}
    Let $\G$ be a Lie group with  a regular, normal, and abelian subgroup $\Ag$. Suppose $\G$ acts freely and properly in a Hamiltonian fashion on
    the symplectic manifold $(M,\omega)$ with equivariant
    momentum map $J_\G:M\to \Lag^*$. Let $\mu\in J_\G(M)\subset \Lag^*$ and suppose  the coadjoint isotropy subgroup $\G_\mu$ is  connected. Then, there is a symplectic diffeomorphism between 
    the symplectic reduced spaces $M\dslash_\mu\G$ and $M\dslash_{i^*(\mu)}\Ag$ if and only if they have the same dimension. 
\end{theorem}

The condition that the reduced symplectic spaces have the same dimension is obviously necessary, so our contribution is
to  show
that it is sufficient. 
This condition only depends on $\G$, $\Ag$, and
$\mu$, and is independent of the symplectic manifold $(M,\omega)$, as follows from Eq.~\eqref{eq:dim-cond} below.

Our proof, presented in Section \ref{s:EqSympRed},
is a consequence of the reduction-by-stages theory developed by Marsden et al.~\cite{Marsden1998,stages}, together with the work of Mykytyuk and Stepin~\cite{Mykytyuk2008663}, which establishes the so-called ``stages hypothesis". The key observation is that  the hypotheses of Theorem~\ref{th:main-new}
imply that the second reduction step in the reduction-by-stages framework is trivial. We also present a direct proof of Theorem~\ref{th:main-new} in Appendix~\ref{app}, 
adapting the constructions in~\cite{stages} to our setting, and
independently establishing the stages hypothesis,
to make our paper self-contained.


Our treatment of examples relies on a technical reformulation of Theorem~\ref{th:main-new}, presented as Proposition~\ref{the:main-1} in Section~\ref{ss:alt-thm}.
This enables us to show that Theorem \ref{th:main-new} applies, for generic $\mu\in \Lag^*$, to a class of metabelian nilpotent Lie groups $\G$ that we call $\Ag$-simple (Definition \ref{def:A-simple}). This class contains several low-dimensional nilpotent groups (see Table \ref{Table:examples}), the Heisenberg groups of arbitrary dimension, and the jet spaces $\J^k(\R^n,\R^m)$. The corresponding result on the symplectic equivalence of the reduced spaces $M\dslash_\mu\G$ and $M\dslash_{i^*(\mu)}\Ag$
for $\Ag$-simple metabelian nilpotent Lie groups is stated and proved in Theorem~\ref{th:generic} of Section~\ref{sss:A-sim}.

\begin{table}[!h]
\begin{center}
    \begin{tabular}{| p{0.5cm} | p{0.8cm} | p{2.5cm} | p{2.5cm} |  p{4.2cm} |}
    \hline
    \tiny{R/D} & 4 & 5 & 6 & 7 \\ \hline
    2 & $N_{4,2}$ & $N_{5,1}$, $N_{5,2,1}$, $N_{5,2,2}$, \fbox{ $N_{5,2,3}$} & $N_{6,1,1}$,   \fbox{$N_{6,1,2}^*$}, $N_{6,1,3}$, $N_{6,2,1}$, \fbox{$N_{6,2,2}^*$}, $N_{6,2,5}$, $N_{6,2,5a}$, $N_{6,2,7}$ & $2457L$, $2457L1$, $2457M$, $23457A$, \fbox{$2347B^*$}, \fbox{$2347C^*$}, \fbox{$12457H^*$}, \fbox{$12457L^*$}, \fbox{$12457L1^*$}, $12457A$ 
     \\ \hline
    3 &  & $N_{5,3,2}$& $N_{6,1,4}$, $N_{6,2,3}$, $N_{6,2,4}$, $N_{6,2,6}$, $N_{6,2,8}$, $N_{6,2,9}$, $N_{6,2,9a}$, $N_{6,2,10}$, $N_{6,3,1}$, $N_{6,3,1a}$, $N_{6,3,2}$, $N_{6,3,3}$, $N_{6,3,4}$, $N_{6,3,5}$, \fbox{$N_{6,3,6}$}, $N_{6,4,4a}$ &  $357A$, $357B$, $247A$, $247B$, $247C$, \fbox{$247D$}, \fbox{$247E$}, \fbox{$247E1$}, \fbox{$247F$}, \fbox{$247F1$}, \fbox{$247G$}, \fbox{$247H$}, \fbox{$247H1$}, \fbox{$247I$}, \fbox{$247J$}, \fbox{$247K$}, $247N$, $247P$, $247P1$, $2457A$, $2457B$, $147D$, $147E$, $147E1$,    \\ \hline
    4 &  & $N_{5,3,1}$ & & $37A$, $37B$, $37B1$, $37C$, $37D$, \fbox{$37D1$}, $257B$, $137A$, $137A1$, $137C$ \\
    \hline
    5 & & & & $27A$, $27B$ \\ \hline
     6 & & & & $17$   \\ \hline
    \end{tabular}
    \vspace{0.2cm}
    \caption{All the stratified Lie algebras up to dimension $7$ and nilpotent Lie algebras up to dimension $6$ were classified in \cite{Cornucopia}. They are listed in the table, where the first column and the first row indicate the rank and dimension of the group, respectively. Our theory applies to those groups whose Lie algebras are {\em not} contained within a box; we refer to them as $\mathbf{\mathbb{A}}$-simple (see Definition \ref{def:A-simple}). Those which are not metabelian have an asterisk in addition to the box.}
 \label{Table:examples}
\end{center}
\end{table}

\subsection{The dimension condition and an example}
The assumption that the action of $\G$ on $M$ is free 
implies that the momentum maps $J_\G:M\to \Lag^*$ and 
$J_\Ag:M\to \Laa^*$ are submersions. As a consequence, 
$\dim J_\G^{-1}(\mu)=\dim M- \dim \G$ and $\dim J_\Ag^{-1}(i^*(\mu))=\dim M- \dim \Ag$. Therefore, 
$$
\dim (M\dslash_\mu\G)= \dim (J_\G^{-1}(\mu)/\G_\mu)=\dim M - \dim \G - \dim \G_\mu,
$$
and, since $\Ag$ is abelian, 
$$
\dim (M\dslash_{i^*(\mu)}\Ag) =\dim (J_\Ag^{-1}(i^*(\mu))/\Ag)=\dim M - 2\dim \Ag. 
$$
Therefore, the condition that the symplectic reduced spaces have the same dimension is
    \begin{equation}
    \label{eq:dim-cond}
    \dim \G+\dim \G_\mu=2\dim \Ag.
    \end{equation}

In Sections~\ref{s:sem-pro} and~\ref{s:nilpotent} we present  examples of  Lie groups $\G$, possessing
a normal abelian subgroup $\Ag$, such that the above condition holds for generic momentum values $\mu\in \Lag^*$. In contrast, note that 
Eq.~\eqref{eq:dim-cond}  holds for the exceptional value $\mu=0$
only in 
the extreme case in which $\G$ and $\Ag$
have the same dimension. 

 The condition \eqref{eq:dim-cond} can be restated in terms
    of the dimension of the coadjoint orbit $\mathcal{O}_\mu$ 
    through $\mu$ as follows. Considering that $\dim \mathcal{O}_\mu=\dim \G - \dim \G_\mu$, we get that Eq.~\eqref{eq:dim-cond}
    is equivalent to 
     \begin{equation}
     \label{eq:dimcondcoadjoint}
    \dim \mathcal{O}_\mu=2(\dim \G-\dim \Ag)=2\dim (\G/\Ag).
    \end{equation}

 Perhaps the simplest nontrivial example is  the
Euclidean group $\G=\SE(2)$
and $\Ag=\R^2$. In this case,  Eq.~\eqref{eq:dim-cond} holds whenever
$\G_\mu$ is one-dimensional which is equivalent
to the condition that the coadjoint orbit
through $\mu\in \mathfrak{se}(2)^*$ be generic. For such coadjoint orbits, the isotropy subgroup $\G_\mu$ is connected (indeed, it
is isomorphic to $\R$), and hence 
Theorem \ref{th:main-new} applies. The theorem implies the somewhat surprising fact that, for generic momentum values, symplectic reduction by the
full Euclidean group $\SE(2)$ is equivalent to 
the symplectic reduction by its translation subgroup $\R^2$.
 The applicability of Theorem~\ref{th:main-new} to other semidirect-product Lie groups is clarified in Theorem~\ref{thm:semidirect}, presented in 
Section~\ref{s:sem-pro}.


\begin{remark}
\label{rmk:2body}
The assertion that the symplectic reduction by $\SE(2)$ is equivalent to the symplectic reduction by $\R^2$ for generic momentum values $\mu\in \mathfrak{se}(2)^*$ seems to be in contradiction with the standard procedure to reduce the planar two-body problem to a 1 degree of freedom Hamiltonian system by first `reducing out  the translations' and afterwards `reducing out  the rotations'. The point is that the 
second reduction 
coincides with the `rotation piece' of
$\SE(2)$ {\em only} if the
the center of mass is located at the origin
of the inertial frame, and this assumption implies that
the momentum $\mu$ is non-generic
(for such $\mu$ the coadjoint orbit
$\mathcal{O}_\mu=\{\mu\}$ is zero-dimensional and $\G_\mu=\SE(2)$ is 
3-dimensional).
For other, generic, values of the momentum $\mu$, the second  reduction 
corresponds to rotations  about the {\em moving}
center of
mass, and this extra rotational-symmetry does {\em not} come
from the $\SE(2)$ action.
It instead comes from  the invariance of the problem with respect to a larger family of symmetries, corresponding to the action of the Galilean group. In fact, it is
the Galilean group symmetry that 
justifies that the common assumption that  the
center of mass is located at the origin of the inertial frame 
is done without loss of generality. 
\end{remark}

\subsection{Usefulness of the construction.} 
A clear application of our result is that abelian reduction is easier to implement. 
For instance, our results imply that for the $N$-vortex problem on the plane with total vanishing circulation, the full symplectic reduction  
by $\SE(2)$ can be obtained by performing the abelian symplectic reduction by $\R^2$ whenever the linear impulse
is non-zero.\footnote{The condition of 
total vanishing circulation  is needed for the momentum map
to be equivariant  \cite{Ohsawa,ModinViviani}. The condition of non-zero linear impulse
is equivalent to saying that the momentum is generic.}
In particular, this equivalence
shows that all non-trivial relative equilibria of the
system having non-vanishing linear impulse are unbounded.

Other applications are the study of sub-Riemannian geodesics on $\Ag$-simple metabelian nilpotent Lie groups, as was done for Engel-type groups in \cite{BravoDoddoli2024}. In fact, the paper \cite{BravoDoddoli2024}, 
which was co-authored by the first author of this work, was
a great motivation for our investigation.

Our work also enables us to understand the structure of the coadjoint orbits of elements $\mu\in \Lag^*$ that satisfy  the hypothesis of Theorem~\ref{th:main-new}. In particular,  the discussion of Section~\ref{ss:coadjoint}, together with  Theorem~\ref{th:generic}, implies that
 the generic coadjoint orbits of $\Ag$-simple metabelian nilpotent
groups are symplectically diffeomorphic to the cotangent bundle $T^*(\G/\Ag)$ equipped
with the canonical symplectic structure plus a magnetic term.\footnote{Using an 
additional ``shift map" \cite{stages}, one can  eliminate the magnetic term and conclude that the generic coadjoint orbits are  symplectically diffeomorphic to the cotangent bundle $T^*(\G/\Ag)$ equipped with its canonical symplectic form.}

\subsection*{Structure of the paper} 
A proof of Theorem
\ref{th:main-new}, which relies on results of \cite{stages,Mykytyuk2008663},
is given in 
Section~\ref{s:EqSympRed}, after the
necessary preliminaries on symplectic
reduction and  notation have been
introduced.  In this section we also 
establish Proposition \ref{the:main-1}, that is fundamental for the
treatment of examples in Section~\ref{s:nilpotent}. Finally, at the end
of Section~\ref{s:EqSympRed}, we 
discuss  the structure of coadjoint orbits
through  the points $\mu\in \Lag^*$ to which Theorem \ref{th:main-new} 
applies. Section~\ref{s:sem-pro} briefly discusses
the case in which $\G$ is a semi-direct
product and Theorem \ref{thm:semidirect} clarifies how Theorem \ref{th:main-new}
applies in this setting.
This  result is not new, as it is
a particular case of the more general theory
of symplectic reduction by 
semi-direct products developed
in \cite{Marsden1998}, but it is included to
illustrate a simple application of 
Theorem \ref{th:main-new}.
In Section~\ref{s:nilpotent}, we introduce $\Ag$-simple groups, which comprise a broad class of nilpotent metabelian Lie groups for which Theorem~\ref{th:main-new} applies for generic $\mu \in \Lag^*$. These groups satisfy the technical condition given in Definition~\ref{def:A-simple}. The main result of this section is  Theorem 
\ref{th:generic}. We also show in this
section that large classes of 
positively graded Lie algebras 
are $\Ag$-simple, including classical
Carnot groups such as the Heisenberg group and the jet space $\J^k(\R^n,\R^m)$.
We conclude the paper by presenting  Appendix~\ref{app}, with a direct
proof of Theorem~\ref{th:main-new}, to make our work self-contained.


\vspace{0.5cm}
We finally mention that a preliminary version of the results of this paper is part of the Master's thesis of the third author \cite{Enzo}, and there is a slight overlap with this work.

\section{Equivalent symplectic reductions}
\label{s:EqSympRed}

This section contains the proof of Theorem \ref{th:main-new} using reduction by stages. We begin by recalling some preliminaries on symplectic reduction and introducing the necessary notation in Subsection~\ref{ss:prelim}, and then present the proof in Subsection~\ref{ss:equivalentred}. In Subsection~\ref{ss:alt-thm}, we give a reformulation of the theorem that will be useful for the treatment of examples in Section~\ref{s:nilpotent}. Finally, we discuss the structure of the coadjoint orbits of the points 
$\mu\in \Lag^*$
 to which Theorem \ref{th:main-new} applies  in Subsection~\ref{ss:coadjoint}.


\subsection{Preliminaries.}
\label{ss:prelim}

Let $(M,\omega)$ be a symplectic manifold. We briefly recall the Marsden-Weinstein-Meyer symplectic reduction scheme  \cite{MW74,meyer}. We refer the reader to the book
\cite{stages} for details and proofs.

Let $\G$ be a Lie group and
suppose that $\G$ defines a free
and proper smooth action on $M$, which will 
be denoted by $(g,m)\mapsto g\cdot m$.
 We say that the action is {\bf \textit{Hamiltonian}}
   if the map $M\to M$, $m\mapsto g\cdot m$, is a symplectic diffeomorphism for
   all $g\in \G$, and the action possesses an equivariant momentum
   map
   \begin{equation*}
       J_\G:M\to \Lag^*,
   \end{equation*}
where $\Lag=\mbox{lie}(\G)$ is the Lie algebra
of $\G$. We recall that the momentum map $J_\G$ is 
characterized by the property that, for any $\xi\in \Lag$, the {\em infinitesimal generator vector field} $\xi_M$ 
on $M$ (defined by $\xi_M(m):=\left . \frac{d}{dt} \right |_{t=0}
\exp(\xi t)\cdot  m$) is a Hamiltonian vector field,
with Hamiltonian function $J_\G^\xi:M\to \R$ given by\footnote{Throughout the paper, given a vector space $V$, we denote by $\langle \cdot, \cdot \rangle_V$ the duality pairing of $V$ and $V^*$.}
 \begin{equation*}
       J_\G^\xi(m):= \langle J_\G(m), \xi\rangle_{\Lag}.
   \end{equation*}
  In other words, we have
 \begin{equation*}
      \iota_{\xi_M}\omega = dJ_\G^\xi, \qquad \mbox{for all 
      $\xi \in \Lag$}.
   \end{equation*}
 Equivariance of 
 $J_\G$ means that 
 \begin{equation*}
       J_\G(g\cdot m)=\Ad_{g^{-1}}^*(J_\G(m)), \qquad\mbox{for all $g\in \G$, $m\in M$,}
   \end{equation*}
where $\Ad^*_{g^{-1}}:\Lag^*\to \Lag^*$ is the coadjoint representation defined by the condition
 \begin{equation*}
      \langle \Ad_{g^{-1}}^*(\mu), \xi \rangle_\Lag = \langle \mu , \Ad_{g^{-1}}(\xi) \rangle_\Lag , \quad \mbox{for all $\mu \in \Lag^*$}, \; \xi \in \Lag. 
   \end{equation*}
   
Under the above conditions, the {\bf \textit{symplectic reduction of $M$ by the 
group $\G$ at $\mu\in \Lag^*$}}, denoted $M\dslash_\mu \G$, is the symplectic manifold $$(J_\G^{-1}(\mu)/\G_\mu,\omega_\mu).$$
Here $\G_\mu$ is the isotropy
subgroup of $\mu$ by the coadjoint
representation, i.e.,
\begin{equation*}
    \G_\mu:=\{g\in \G\, :\, \Ad_{g^{-1}}^*(\mu)=\mu \, \},
\end{equation*}
    and the symplectic form $\omega_\mu$ on the quotient space 
   $J_\G^{-1}(\mu)/\G_\mu$ is characterized
    by the condition
    \begin{equation}
    \label{eq:symp-red-form}
 \pi_\mu^*\omega_\mu=j^*_\mu \omega,
\end{equation}
where $\pi_\mu:J^{-1}(\mu)\to J^{-1}(\mu)/\G_\mu$ is the orbit projection and
$j_\mu:J^{-1}(\mu)\hookrightarrow M$ is the 
inclusion.

Now suppose that $\Ag\subset \G$ is 
a closed, abelian, normal subgroup. The Lie algebra $\Laa$ of $\Ag$ is then
an abelian ideal of $\Lag$.  The inclusion $\Ag\hookrightarrow \G$ induces a
Lie algebra inclusion $i:\Laa \hookrightarrow \Lag$ whose dual,
\begin{equation}
    i^*:\Lag^* \to \Laa^*,
\end{equation}
plays an important role in our work. Note that
for $\mu\in \Lag^*$ we have $i^*(\mu)=\left . \mu \right |_\Laa$. 

It is not difficult to show (see e.g. \cite{stages}, \cite{Enzo}) that
the   $\G$-action
on $M$ restricts  to a free and proper $\Ag$-action
on $M$, which is again Hamiltonian,
 and whose equivariant momentum map $J_\Ag:M\to \Laa^*$ is given by
 \begin{equation}
 \label{eq:JAJG}
        J_\Ag=i^* \circ J_\G.
   \end{equation}
In particular, we may consider the
symplectic reduction of $M$ by the group
$\Ag$ at $\nu \in \Laa^*$, denoted $M\dslash_\nu \Ag$. Given that $\Ag$ is abelian, the
coadjoint representation of $\Ag$ on $\Laa^*$ is trivial and the coadjoint isotropy $\Ag_\nu=\Ag$ for all
$\nu\in \Laa^*$. Therefore, $M\dslash_\nu \Ag$
is the symplectic manifold
$$(J_\Ag^{-1}(\nu)/\Ag,\omega_\nu),$$
with $\omega_\nu$ characterized in analogy with Eq.~\eqref{eq:symp-red-form}.

\subsection{Proof of Theorem \ref{th:main-new} using reduction-by-stages}
\label{ss:equivalentred} 
Given $\mu\in \Lag^*$, recall that the Lie subalgebra $\Lag_{\mu}$ of  $\G_{\mu}$ is  given by
$$ \Lag_{\mu} =\left \{ W \in \Lag \, : \,  \langle \mu, [W,\tilde W]\rangle_\Lag = 0\;\text{for all}\;\;\tilde W \in \Lag  \right \}.  $$

For $\mu\in \Lag^*$, consider the skew-symmetric bilinear form 
\begin{equation}
    \label{eq:def-Omegamu}
    \Omega_{\mu}:\Lag\times \Lag \to \R, \qquad  \Omega_{\mu}(W_1,W_2):=\langle \mu , [W_1,W_2] \rangle_\Lag.
\end{equation}

 It is immediate to see that 
\begin{equation*}
    \begin{split}
      \Lag_{\mu}  = \ker \Omega_{\mu}.
    \end{split}
\end{equation*}

\begin{proposition}
\label{prop:a-isotropic}
The dimension condition \eqref{eq:dim-cond} holds
if and only if $\Laa$ is a maximal isotropic subspace of~ $\Omega_{\mu}$.
In particular, if Eq.~\eqref{eq:dim-cond} holds, then $\Lag_\mu\subset \Laa$.
\end{proposition}
\begin{proof}
    Given a skew-symmetric bilinear form $B:V\times V\to \R$ on a real finite-dimensional vector space $V$, maximal isotropic subspaces always exist and have the same dimension. Indeed, the rank of $B$ is always even, and the dimension of maximal isotropic subspaces is half the rank of $B$ plus the dimension of the kernel of $B$. 

    Therefore, the dimension of the maximal isotropic subspace of $\Omega_{\mu}$ is given by
\begin{equation*}
    \begin{split}
      \frac{1}{2} \mbox{rank} \;\Omega_{\mu} + \dim \ker \Omega_{\mu} & = \frac{1}{2}(\dim \Lag - \dim \Lag_{\mu}) + \dim \Lag_{\mu} \\
      & = \frac{1}{2}(\dim \Lag + \dim \Lag_{\mu}). 
    \end{split}
\end{equation*}
Now,  since $\Laa$ is an abelian subalgebra, it is an isotropic subspace. On the other hand, the dimension  condition \eqref{eq:dim-cond} is equivalent to  
$\dim \Laa=\frac{1}{2}(\dim \Lag + \dim \Lag_{\mu})$, which shows that the dimension of $\Laa$ is that  of maximal isotropic subspaces of $\Omega_\mu$ if and only if Eq.~\eqref{eq:dim-cond} holds. 

The conclusion that $\Lag_\mu$ is contained in $\Laa$ when Eq.~\eqref{eq:dim-cond} holds, follows immediately since the maximal isotropic
subspace $\Laa$ must contain $\ker \Omega_\mu=\Lag_\mu$.
\end{proof}

 Consider the representation of
 $\G$ on $\Laa^*$ defined by 
\begin{equation}
\label{eq:stages-rep}
    \langle g\cdot \nu,Y  \rangle_\Laa := \langle \nu ,
 \Ad_{g^{-1}}Y \rangle_\Laa, \qquad \nu\in \Laa^*, \; Y\in \Laa, \; g\in \G.
\end{equation}
  It 
 is well-defined since $\Ad_\G(\Laa)\subset \Laa$
 by normality of $\Ag$. Denote by $\G_\nu$ the isotropy subgroup of $\nu\in \Laa^*$ and by $\Lag_\nu$ its Lie algebra. 
 
 \begin{proposition}
 \label{prop:stages-aux}Let $\mu\in \Lag^*$, and write  $\nu=i^*(\mu)$.
     The dimension condition \eqref{eq:dim-cond} holds
     if and only if $\Lag_{\nu}=\Laa$.
 \end{proposition}
 
\begin{proof}
 The Lie algebra representation associated to Eq.~\eqref{eq:stages-rep} satisfies
    \begin{equation*}
       \langle  W\cdot \nu ,Y \rangle_\Laa =-\langle \nu, [W,Y]\rangle_\Laa, \qquad W\in \Lag, \;  Y\in \Laa.
    \end{equation*}
 Therefore, $W\in \Lag_{\nu}$ if and only if
    $ \langle  \nu ,[W,Y] \rangle_\Laa =0$ for all
    $Y\in \Laa$, or equivalently, since $\Laa\subset \Lag$ is an ideal, if
    \begin{equation*}
       \langle  \mu ,[W,Y] \rangle_\Lag =\Omega_\mu(W,Y)=0, \quad \mbox{for all} \quad Y\in \Laa.
    \end{equation*}
    It is immediate to see that $\Laa\subset \Lag_{\nu}$ by the abelianity of
$\Laa$.
On the other hand, again by the abelianity of $\Laa$,
it follows that $W\in \Lag_{\nu}$ if
and only if $\Laa+\langle W\rangle$ is an isotropic subspace of $\Omega_\mu$.
But this implies that
$\Lag_{\nu}\subset \Laa$ if and only
if $\Laa$ is a maximal isotropic subspace
of $\Omega_\mu$, and the result follows
from Proposition \ref{prop:a-isotropic}.
\end{proof}

The above proposition, together with the
abelianity of $\Ag$, implies that  the ``stages hypothesis",
proved in \cite[Remark 4]{Mykytyuk2008663},
admits the following formulation 
when the dimension condition \eqref{eq:dim-cond}
holds.\footnote{It is shown in \cite{Mykytyuk2008663}
that the group element $a$ in the statement
of  Lemma~\ref{lem:StagesMyk} actually belongs to the connected
component of the identity, $\Ag^0$. We will have
no need of this extra information.}

\begin{lemma}[Stages Hypothesis \cite{Mykytyuk2008663}] 
\label{lem:StagesMyk}
Suppose that the dimension condition 
\eqref{eq:dim-cond} holds for $\mu\in \Lag^*$, and write $\nu=i^*(\mu)$. 
If $\tilde \mu\in \Lag^*$ satisfies $ i^*(\tilde \mu)=\nu$, then there exists  $a\in \Ag$ such that 
$$ \mu = \Ad^*_{a^{-1}}(\tilde \mu).$$ 
\end{lemma}

The following proposition is 
essential in the proof of Theorem~\ref{th:main-new} that we give below.

\begin{proposition}
\label{prop:trivialstages}
    Suppose that $\mu\in \Lag^*$ satisfies the dimension condition \eqref{eq:dim-cond}, and write $\nu=i^*(\mu)$.
    \begin{enumerate}
        \item[{\em(a)}] If  $\G_\mu$ is connected, then  $\G_\mu\subset \Ag$.
        \item[{\em(b)}] If  $\G_\mu\subset \Ag$, then  $\G_\nu=\Ag$.
    \end{enumerate}
    In particular, under the hypothesis of Theorem \ref{th:main-new},   
$$\G_\nu=\Ag.$$
\end{proposition}
\begin{proof}
By Proposition \ref{prop:a-isotropic}, the
assumption that the dimension condition \eqref{eq:dim-cond}
holds implies 
    $\Lag_\mu\subset \Laa$.
Therefore,
since $\G_\mu$ is connected, 
\begin{equation}
\label{eq:GmucontA}
    \G_\mu\subset \Ag,
\end{equation}
and item (a) is proved.

 Let $ a\in \Ag$ and $Y\in \Laa$. In view of Eq.~\eqref{eq:stages-rep}, we have
\begin{equation*}
    \langle a\cdot \nu,  Y\rangle_\Laa= \langle  \nu,\Ad_{a^{-1}}  Y\rangle_\Laa =\langle \nu, Y\rangle_\Laa,
\end{equation*}
where the second equality follows from the abelianity of
$\Ag$. Hence, $a\cdot \nu=\nu$ and $a\in \G_\nu$, proving
that $\Ag\subset\G_\nu$.

Now, let $g\in \G_\nu$ and $Y\in \Laa$. We have
\begin{equation*}
    \langle i^*(\Ad_{g^{-1}}^*\mu), Y\rangle_{\Laa}=\langle \Ad_{g^{-1}}^*\mu,Y\rangle_\Lag=\langle \mu,\Ad_{g^{-1}}Y\rangle_\Lag. 
\end{equation*}
By normality of $\Ag$, we have $\Ad_{g^{-1}}Y\in \Laa$.
Hence, $\langle \mu,\Ad_{g^{-1}}Y\rangle_\Lag=\langle \nu,\Ad_{g^{-1}}Y\rangle_\Laa$, and
we conclude that 
\begin{equation*}
    \langle i^*(\Ad_{g^{-1}}^*\mu), Y\rangle_{\Laa}=\langle \nu,\Ad_{g^{-1}}Y\rangle_\Laa=\langle g\cdot \nu,Y\rangle_\Laa=\langle \nu, Y\rangle_\Laa. 
\end{equation*}
Since $Y\in \Laa$ is arbitrary, this shows that
\begin{equation*}
    i^*(\Ad_{g^{-1}}^*\mu)= \nu. 
\end{equation*}
Therefore, by Lemma \ref{lem:StagesMyk}, there exists $a\in \Ag$ such that $\Ad^*_{a^{-1}}(\Ad_{g^{-1}}^*\mu)=\mu$,
or, equivalently
\begin{equation*}
    \Ad^*_{(ag)^{-1}}\mu=\mu.
\end{equation*}
Thus, $ag\in \G_\mu$, and hence the condition that $\G_\mu\subset \Ag$ implies 
$$g\in \Ag,$$ proving that $\G_\nu\subset \Ag$.
\end{proof}

We are now ready to present the proof of Theorem~\ref{th:main-new}.\footnote{For completeness, self-contained proofs of Lemma~\ref{lem:StagesMyk} and
Theorem~\ref{th:main-new} are given in Appendix~\ref{app}.
}

\begin{proof}[Proof of Theorem~\ref{th:main-new}]
    Let $\mu\in \Lag^*$ and denote $\nu=i^*(\mu)$. 
By abelianity of $\Ag$, the $\Ag$-coadjoint isotropy subgroup, $\Ag_\nu=\Ag$.
The theory of reduction-by-stages  \cite[Theorem 5.2.9]{stages}, complemented with the stages hypothesis, presented above as Lemma \ref{lem:StagesMyk}, implies that there exists a symplectic diffeomorphism between the
``one-shot-reduced-space" $M\dslash_\mu\G$ 
and the symplectic reduction of ``the intermediate-reduced-space" $M\dslash_{\nu}\Ag$ by the group $\mathbb{K}_\nu:=\G_\nu/\Ag_\nu=\G_\nu/\Ag$ (at an adequate momentum value).  Proposition \ref{prop:trivialstages}  implies that, under the assumptions of Theorem \ref{th:main-new}, 
we have $\G_\nu=\Ag$ so the Lie group $\mathbb{K}_\nu$ is
trivial. Therefore, the second
reduction is trivial and $M\dslash_\mu\G$ 
and $M\dslash_{\nu}\Ag$ are symplectomorphic.
\end{proof}

\begin{remark}
    If one drops the hypothesis that $\G_\mu$ is connected
    from Theorem \ref{th:main-new}, the  proof
of Proposition \ref{prop:trivialstages}    that shows that $\G_\nu=\Ag$ is no longer valid and one cannot
conclude that the group $\mathbb{K}_\nu=\G_\nu/\Ag$ is trivial. However, Proposition \ref{prop:stages-aux} implies that
$\mathbb{K}_\nu$ is discrete, and therefore  reduction-by-stages 
implies the existence of a smooth symplectic covering map from $M\dslash_{\nu}\Ag$ onto $M\dslash_\mu\G$. It is possible to 
build explicit examples in which this covering map is 
not a diffeomorphism. In particular, this shows that the
hypothesis that $\G_\mu$ is connected in Theorem \ref{th:main-new} cannot be dropped.
\end{remark}

Finally, we present the following proposition, which  is useful
for reformulating  Theorem \ref{th:main-new} in the next section.

\begin{proposition}
\label{prop:GmusubsetA} 
    Theorem \ref{th:main-new} remains valid replacing the
    hypothesis that $\G_\mu$ is connected by the 
    condition that $\G_\mu\subset \Ag$.
\end{proposition}
\begin{proof}
    Item (b) of  Proposition \ref{prop:trivialstages}  shows
    that the dimension condition \eqref{eq:dim-cond}, together with 
    the set inclusion $\G_\mu\subset \Ag$, imply $\G_\nu=\Ag$.
    Therefore, the group $\mathbb{K}_\nu$ of the second step of the
    reduction by stages is also trivial  in this case.
\end{proof}




\subsection{Alternative formulation in terms of \texorpdfstring{\(\T_\mu\)}{} }
\label{ss:alt-thm}

Denote by $\Z(\Lag)$  the center of $\Lag$.
 
\begin{proposition}
\label{prop:Zsubseta}
Let $\mu\in \Lag^*$ and suppose that the dimension condition  \eqref{eq:dim-cond} holds,
 then
 $$\Z(\Lag)\subset \Lag_{\mu} \subset \Laa.$$ 
\end{proposition}
\begin{proof}
We had already established the validity of  $\Lag_\mu\subset \Laa$ whenever Eq.~\eqref{eq:dim-cond} holds in Proposition \ref{prop:a-isotropic}.
   On the other hand, consider the bilinear form 
   $\Omega_\mu$ defined by Eq.~\eqref{eq:def-Omegamu}. Clearly $\Z(\Lag)\subset \ker\Omega_\mu=\Lag_\mu$. 
\end{proof}

In view of the above proposition, 
in the framework of Theorem \ref{th:main-new}, we have 
  $$\Z(\Lag)\subset \Laa.$$
  The following considerations are made under the assumption that this set inclusion holds.
  
  Consider the quotient space
$\Laa/\Z(\Lag)$ and its dual $\left ( \Laa/\Z(\Lag) \right )^*$.  
In order to realize the latter space in less abstract terms, we consider the decomposition
\begin{equation}
\label{eq:glag-desc1}
    \Lag = \mathfrak{X} \oplus \Laa,
\end{equation}
where 
$\mathfrak{X}$ is  any direct complement of $\Laa$ in $\Lag$.
This induces a dual decomposition in terms of the annihilator
spaces,\footnote{If $U$ is a subspace of a vector
space $V$, its annihilator, $U^\circ$, is the 
subspace of $V^*$ formed by the elements $\alpha \in V^*$ such that $\langle \alpha, u\rangle_V =0$ for all $u\in U$.}
\begin{equation}\label{eq:glagdual-desc}
\Lag^* = \Laa^\circ \oplus\mathfrak{X}^\circ,    
\end{equation}
 and we  may identify 
 \begin{equation}
 \label{eq:dual-ids}
     \mathfrak{X}^* \cong \Laa^{\circ}, \qquad \Laa^*\cong \mathfrak{X}^\circ.
 \end{equation}
  With these  identifications we  think of $\Laa^*=\mathfrak{X}^\circ$ as a subspace of $\Lag^*$ and  there
is a natural linear isomorphism $$\poj :(\Z(\Lag))^{\circ}\cap\Laa^*\to (\Laa/\Z(\Lag))^*.$$
If $\nu \in (\Z(\Lag))^{\circ}\cap\Laa^*$, then   $\poj ({\nu}) \in (\Laa/\Z(\Lag))^*$ is defined  by  
$$\langle\poj({\nu}),\overline{Y}\rangle_{\Laa/\Z(\Lag)}: = \langle\nu,Y\rangle_{\Laa},$$
where $\overline{Y} \in \Laa/\Z(\Lag)$ denotes the image 
of $Y\in \Laa$ under the quotient projection $\Laa\to \Laa/\Z(\Lag)$. 

The following linear map, defined for fixed
$\mu\in \Lag^*$, is central to our work: 
    \begin{equation}
    \label{eq:defTmu}
   \T_{\mu}:\mathfrak{X} \to (\Laa/\Z(\Lag))^*, \qquad       \T_{\mu}(X) := \poj(i^*(\ad^*_{-X}(\mu))).
    \end{equation}
The following proposition shows that the
dimension condition \eqref{eq:dim-cond} may be
encoded in terms of the injectivity of $\T_{\mu}$.

\begin{proposition}
\label{prop:Tmu}
For any $\mu\in \Lag^*$, the linear map $\T_{\mu}:\mathfrak{X} \to (\Laa/\Z(\Lag))^*$ is well-defined. Moreover,  the 
dimension condition \eqref{eq:dim-cond} holds if and only if~  $\T_{\mu}$ is injective. 
\end{proposition}

\begin{proof}
    Since $i^*$ has domain $\Lag^*$ and codomain $\Laa^*$, it is clear that $i^*(\ad^*_{-X}(\mu)) \in\Laa^*$ for any $X\in \mathfrak{X}$. Moreover, if $Z\in\Z(\Lag)$ is arbitrary then 
    \begin{equation*}
     \begin{split}
        \langle i^*(\ad^*_{-X}(\mu)),Z\rangle_{\Laa}& = \langle\ad^*_{-X}(\mu),i(Z)\rangle_{\Lag} \\
        & =  \langle\mu,\ad_{-X}(Z)\rangle_{\Lag} = 0,\\
    \end{split}
    \end{equation*}
    where we have used $\ad_{-X}(Z) = 0$ since $Z\in \Z(\Lag)$. Therefore,   $i^*(\ad^*_{-X}(\mu))$ belongs to the domain, $ (\Z(\Lag))^{\circ} \cap \Laa^*$, of   $\poj$. In  view of the definition of $\poj$, we conclude that $\T_{\mu}$ is well-defined as a mapping from $\mathfrak{X}$ into $(\Laa/\Z(\Lag))^*$.

Now let $X\in \mathfrak{X}$ and $Y\in \Laa$. We
may write 
\begin{equation*}
    \begin{split}
       \langle\T_{\mu}(X),\overline{Y}\rangle_{\Laa/\Z(\Lag)}&= \langle i^*(\ad^*_{-X}(\mu)),Y \rangle_{\Laa} = \langle \ad^*_{-X}(\mu),Y \rangle_{\Lag}\\
      &  = \langle \mu,-[X,Y] \rangle_{\Lag}=-\Omega_\mu(X,Y).
    \end{split}
\end{equation*}
Hence, $X\in \ker \T_\mu$ if and only if $\Omega_\mu(X,Y)=0$
for all $Y\in \Laa$. Considering that $\Laa$ is an isotropic
subspace of $\Omega_\mu$, we conclude that $X\in \ker \T_\mu$ if and only if $\Laa+\langle X\rangle$ is isotropic.
Therefore, $\ker \T_\mu=\{0\}$ if and only if $\Laa$
is a maximally isotropic subspace of $\Omega_\mu$, and by Proposition \ref{prop:a-isotropic} this implies that $\T_\mu$ is injective
if and only if the dimension condition \eqref{eq:dim-cond} holds.

\end{proof}

      The following reformulation of Theorem \ref{th:main-new}
      will be used in Section \ref{s:nilpotent}. It is
      a direct consequence of Propositions~\ref{prop:GmusubsetA} and~\ref{prop:Tmu}.

\begin{proposition}\label{the:main-1}
         Let $\G$ be a Lie group with a regular, normal, and abelian subgroup $\Ag$. Suppose $\G$ defines a 
    free and proper Hamiltonian action  on
    the symplectic manifold $(M,\omega)$ with equivariant
    momentum map $J_\G:M\to \Lag^*$. Let $\mu\in J_\G(M)\subset \Lag^*$ and suppose that $\G_\mu \subset \Ag$. Let
    $\mathfrak{X}$ be any complement of $\Laa$ in $\Lag$ and
    consider the map\footnote{note that the quotient space $\Laa/\Z(\Lag)$ 
    is well-defined since  $\Z(\Lag)\subset\Laa$. This is true
    because $\Z(\Lag)\subset \Lag_\mu $ always
    holds and $\Lag_\mu\subset \Laa$ by the assumption that $\G_\mu\subset \Ag$.} $\T_{\mu}:\mathfrak{X} \to (\Laa/\Z(\Lag))^*$ 
    defined by Eq.~\eqref{eq:defTmu}.
     There is a symplectic diffeomorphism between 
    the symplectic reduced spaces $M\dslash_\mu\G$ and $M\dslash_{i^*(\mu)}\Ag$ if and only if $\T_\mu$
    is injective. 
\end{proposition}
\subsection{Calculation of coadjoint orbits} 
\label{ss:coadjoint}
Take $M=T^*\G$ in Theorem \ref{th:main-new} and consider the action of $\G$ on $T^*\G$ by the
cotangent lift of 
left multiplication.  This action is free, proper and Hamiltonian,
and has 
surjective momentum map $J_\G:T^*\G\to \Lag^*$ 
given by the right trivialisation. As is well-known, for $\mu\in \Lag^*$,  the symplectic reduced space $M\dslash_\mu \G$ is symplectically 
diffeomorphic to the coadjoint orbit $(\mathcal{O}_\mu, \omega_\mu)$ where $\omega_\mu$
denotes the (minus) Kostant-Kirillov-Soriau symplectic form \cite{MW74}. Therefore, 
if $\G_\mu$ is
connected, and the dimension condition \eqref{eq:dim-cond} holds, then Theorem \ref{th:main-new}
implies that $T^*\G\dslash_\mu \Ag$ is symplectically diffeomorphic to
$(\mathcal{O}_\mu, \omega_\mu)$.  On the other hand, by the
theory of cotangent bundle reduction (see e.g. \cite[page 300]{AbrahamMarsden} or \cite[Theorem 6.6.3]{OrtegaRatiu}), we know that 
$T^*\G\dslash_\mu \Ag$ is symplectically diffeomorphic to the 
cotangent bundle $T^*(\G/\Ag)$ equipped with the symplectic form
$\omega_{\G/\Ag}-B_\mu$, where $\omega_{\G/\Ag}$ is the canonical 
symplectic form and $B_\mu$ is a {\em magnetic term}. Therefore, we conclude that the coadjoint orbits of those $\mu\in \Lag^*$
for which Theorem \ref{th:main-new} applies are symplectically
diffeomorphic to $T^*(\G/\Ag)$ (equipped with the symplectic form
$\omega_{\G/\Ag}-B_\mu$).

We mention that conclusions of the above type 
are well-known for the coadjoint orbits of semidirect product groups,
without requiring that the condition \eqref{eq:dim-cond} holds. 
For instance, see the ``Semi-direct product reduction theorem" in
\cite{Marsden1998}. This type of correspondence is due to the
relationship of our work with the theory of reduction by stages.




\section{Examples I: Semidirect products}\label{s:sem-pro}

 As a first class of examples, suppose that $\Ag$ is 
a real vector space and let $\G=\Hg \ltimes \Ag$ be the 
 semidirect product 
 group, where $\Hg$ is  a connected Lie group 
 acting linearly on $\Ag$. As a manifold, $\G=\Hg\times \Ag$ and the group operation on $\G$ is
\begin{equation*}
    (h_1,a_1)(h_2,a_2)=(h_1h_2,a_1+h_1a_2).
\end{equation*}
Then $\Ag$ is a regular, normal, abelian subgroup of $\G$.

Symplectic reduction by this class of groups has
been extensively studied (see e.g. \cite[Section 4.2]{stages} for a comprehensive list of references). The result
that we present below as Theorem \ref{thm:semidirect} is a particular
instance of  the ``Reduction by Stages for Semidirect Product Actions" in \cite[Theorem 3.2]{Marsden1998}. In this sense, our discussion here does not give any new results but attempts only to illustrate 
how Theorem  \ref{th:main-new} applies in the simple context
of semidirect products. 

The Lie algebra $\Lag$ is the semidirect product 
Lie algebra $\Lah\ltimes \Ag$, where $\Lah$ is the
Lie algebra of $\Hg$ and we have identified $\Laa=\Ag$. The bracket is
given by
\begin{equation*}
    (\xi_1,a_1)(\xi_2,a_2)=([\xi_1,\xi_2],\xi_1a_2-\xi_2a_1),
\end{equation*}
where $\xi a$ denotes the induced action of $\xi \in \Lah$ on
$a\in \Ag$.

For the formulation of Theorem \ref{thm:semidirect} note that, as a vector space, the dual Lie algebra $\Lag^*=\Lah^*\times \Ag^*$. Moreover,  the mapping  $i^*:\Lag^*=\Lah^*\times \Ag^*\to \Ag^*$ is given by $i^*(\gamma,\nu)=\nu$. 
Also note that there is an induced action of $\Hg$ on $\Ag^*$, by the action
of the transpose of $h^{-1}$ on $\Ag$. The corresponding isotropy group of $\nu\in \Ag^*$
will be denoted by  $\Hg_\nu$.

\begin{theorem}
\label{thm:semidirect}
    As above, let $\G$ be the semidirect product $\Hg\ltimes \Ag$, with $\Hg$ connected and $\Ag$ a
    vector space.
    Suppose that $\G$ defines a 
    free and proper Hamiltonian action on the symplectic manifold $(M,\omega)$ with equivariant
    momentum map $J_\G:M\to \Lah^*\times \Ag^*$. Let     $\mu=(\gamma, \nu)\in J_\G(M)\subset \Lah^*\times \Ag^*$. There exists a symplectic diffeomorphism
between the  symplectic reduced spaces $M\dslash_\mu\G$ and $M\dslash_{\nu}\Ag$ if and only if
$\Hg_\nu=\{e_\Hg\}$. 
\end{theorem}

 The proof proceeds by showing that the condition that $\Hg_\nu=\{e_\Hg\}$ 
is equivalent to the simultaneous validity 
of the following two conditions: 1. the dimension of the
coadjoint 
orbit $\mathcal{O}_\mu$ satisfies Eq.~\eqref{eq:dimcondcoadjoint}; 2. the subgroup $\G_\mu$ is
connected. After establishing this equivalence,
the proof follows from
 Theorem \ref{th:main-new}. We omit the details.\footnote{In fact, Theorem 3.1
in \cite{Marsden1998} shows that if $\Hg_\nu=\{e_\Hg\}$ then
$\mathcal{O}_\mu$ is symplectically
diffeomorphic to $T^*\Hg$ and in particular its dimension satisfies Eq.~\eqref{eq:dimcondcoadjoint}.}

As indicated in the introduction, Theorem \ref{th:main-new} 
applies to
the Euclidean group $\G=\SE(2)$ (for generic elements $\mu\in \Lag^*$). Let us see why this is true in the light of Theorem
\ref{thm:semidirect}. We have $\G=\SE(2)=\SO(2)\ltimes \R^2$ where $\Hg=\SO(2)$ acts on $\Ag=\R^2$ by
rotations. Upon identification of $\Ag^*$ with $\R^2$ via the
dot product, the resulting action of $\SO(2)$ on $\Ag^*$ 
is again by rotations. A nonzero $\nu\in \R^2=\Ag^*$ is not fixed by 
any nontrivial rotation so $\SO(2)_\nu=\{e_{\SO(2)}\}$ and Theorem
\ref{thm:semidirect} can be used to conclude the equivalence
between the abelian and nonabelian symplectic reductions.
 


\section{Examples II: A class of metabelian nilpotent groups}
\label{s:nilpotent}
In this section, we present a class of nilpotent groups to which Theorem~\ref{th:main-new} applies for generic $\mu \in \Lag^*$. We begin by recalling the necessary preliminaries, including the definitions of nilpotent, Carnot, and metabelian groups in Subsection~\ref{sss:prel}. Then,
in Subsection~\ref{sss:A-sim} we give the   definition of $\Ag$-simple group (Definition~\ref{def:A-simple}) and we state and
prove  Theorem~\ref{th:generic}, which is the main result of this section. The theorem
 establishes that, for an $\Ag$-simple group $\G$, there exists an open and dense subset $\Lagr \subset \Lag$ with the property that the reduced spaces $M\dslash_\mu\G$ and $M\dslash_{i^*(\mu)}\Ag$ are symplectomophic for all $\mu \in \Lagr$. 
Finally, in Subsection~\ref{sss:A-sim-exa} we show that 
$\Ag$-simple groups are plentiful within the class of metabelian nilpotent groups. Low-dimensional examples
are indicated in Table \ref{Table:examples}
in the introduction. Other notable examples include the Heisenberg group, Carnot groups whose center is one-dimensional, and the jet space $\J^k(\R^n,\R^m)$.

Throughout this section, we assume that $\G$ is a connected and simply connected meta-belian nilpotent  Lie group with Lie algebra $\Lag$.  As a consequence, the exponential map $\exp :\Lag \to \G$ is a global diffeomorphism (see e.g. \cite[Theorem 1.2.1]{corwin1990representations}). This will 
allow us to deduce several properties
of $\G$ from properties of its Lie algebra
$\Lag$.

\subsection{Preliminaries}
\label{sss:prel}

\subsubsection{Nilpotent groups} Given a subgroup $\mathbb{H}$ of $\G$, we denote by $[\mathbb{H},\G]$ the subgroup generated by all the commutators, i.e., 
$$ [\mathbb{H},\G] := \left \langle \{ hgh^{-1}g^{-1}: h \in \mathbb{H}\;\text{and}\;g \in \G\} \right \rangle. $$
The {\bf \textit{descending central series}} of $\G$ is given by
$$\G^1:= \G, \;\; \G^{i+1}:= [\G,\G^{i}]\;\;\text{for all}\;\;i=1,2,\dots.$$
We say that $\G$ is {\bf \textit{nilpotent}} if there exists $s\in \mathbb{N}$ such that $\G^{s+1} = \{e_{\G}\}$. If
$\G^{s}\neq \{e_{\G}\}$ then we call $s$ the {\bf \textit{step}} of $\G$. Let us denote the group's center by $\Z(\G)$. If $\G$ is a nilpotent group with step $s$, it follows that $\G^{s} \subset \Z(\G)$. Therefore, the center $\Z(\G)$ is not empty for every nilpotent group.

The connectedness of  $\G$ allows us to characterize the nilpotency condition in
terms of its Lie algebra $\Lag$.  Given two subspaces $\V_1, \V_2 \subset \Lag$, we denote by $[\V_1,\V_2]$ the subspace of all their Lie brackets, i.e.,
$$ [\V_1,\V_2] = \left \{\,[W_1,W_2]:  W_1\in \V_1 \;\;\text{and}\;\;W_2\in \V_2 \,\right \}.$$
The {\bf \textit{descending central series}} of $\Lag$ is defined by
$$ \Lag^1 := \Lag,\qquad \Lag^{i+1} := [\Lag,\Lag^{i}],\quad \text{for all $i=1,2,\dots $.}$$
 It is well-known that $\Lag^i$ is the
 Lie algebra of 
the $i^{th}$ element, $\G^i$, of the descending
central series of $\G$ \cite[Theorem 12.3.1]{hochschild1969structure}. In particular,
if the Lie group 
$\G$ is nilpotent of step $s$, we have
$\{0\}\neq \Lag^{s}\subset \Z(\Lag)$ and 
$\Lag^{s+1} = \{0\}$. In this case,
we say that the Lie algebra $\Lag$ is {\bf \textit{nilpotent}} of {\bf \textit{step}} $s$. As in the group case,  since 
$\Lag^{s}\subset \Z(\Lag)$,  we conclude that every nilpotent Lie algebra has a non-trivial center. 
Conversely, since $\G$ is assumed to be  connected, the condition that $\Lag$ is nilpotent of step $s$ implies
that the $\G$ is nilpotent of step $s$.

A particular example of a nilpotent group is a {\bf \textit{Carnot}} group. A Carnot group is a connected and simply connected nilpotent Lie group whose Lie algebra is stratified. A {\bf \textit{stratification}}  of a nilpotent Lie algebra $\Lag$ of step $s$ is a direct-sum decomposition 
$$ \Lag = \V_1 \oplus \V_2\oplus \cdots \oplus \V_s,$$
 where $\V_s\neq \{0\}$ and $[\V_1,\V_a] = \V_{1+a}$  for $a=1,\dots, s$, and $\V_{s+1} = \{0\}$.   The subspaces $\V_a$ are called the {\bf \textit{layers}} of $\Lag$. We say a Lie algebra is {\bf \textit{stratifiable}} if a stratification exists. Although stratifications are not unique,
any two stratifications of $\Lag$ differ by a Lie algebra automorphism  (see \cite[Proposition 2.17]{LeDonne+2017+116+137}).

Carnot groups constitute a fundamental subclass of nilpotent groups. In the framework of sub-Riemannian geometry, they play a role analogous to that of Euclidean spaces in Riemannian geometry, as the tangent space of a sub-Riemannian manifold carries the structure of a Carnot group, see \cite{gromov1996carnot} and also \cite{mitchell1985carnot,bellaiche1996tangent}.

A more general class of nilpotent Lie algebras consists of those that are positively gradable.  A {\bf \textit{positive grading} }of a Lie algebra is a family $(\V_a)_{a \in \mathbb{N}}$ of linear subspaces of $\Lag$ with the property that only finitely many of them are nonzero, and such that the following conditions
are satisfied: 
$$ \Lag = \bigoplus_{a \in \mathbb{N}} \V_a \qquad \mbox{and} \qquad [\V_a,\V_b]  \subset \V_{a+b} \;\; \mbox{for all}\;\; a,b\in \mathbb{N}.$$
 We say that a Lie algebra is {\bf \textit{positively gradable}} if it admits a positive grading. Every stratifiable algebra is positively gradable, but the converse is not true \cite{le2019besicovitch}. 

 A subspace $\V\subset \Lag$ is said to be {\bf \textit{ bracket generating }} if every vector in $\Lag$ 
 can be written as a linear combination of a finite number, $k\geq 0$, of bracket iterations of elements in $\V$.
A subspace $\V$ is bracket generating if and only if 
$$ \V + [\Lag,\Lag] = \Lag. $$
The {\bf \textit{rank}} of a positively gradable algebra is the dimension of the smallest bracket generating subspace. In particular,  for Carnot 
groups, the first layer, $\V_1$, is bracket generating and the rank of the group is $\dim\V_1$. In \cite{Cornucopia}, Le Donne and Tripaldi classified, in terms of their dimension and their rank, the stratifiable Lie algebras of dimension up to
7, and the positive gradable  Lie algebras of dimension up to
6 (see Table \ref{Table:examples}).


\subsubsection{Metabelian nilpotent groups}\label{subsubs:meta}

To the best of our knowledge, the notion of metabelian groups
was introduced by Robinson (see e.g. \cite{robinson2012course}). A group $\G$ is {\bf \textit{metabelian}} if $[\G,\G]$ is abelian. Considering that $[\G,\G]$ is a normal subgroup of $\G$ and $\G/[\G,\G]$ is abelian,  it follows that metabelian groups always possess an abelian normal subgroup $\mathcal{N}$ such that $\G/\mathcal{N}$ is abelian. Conversely, if a group $\G$ posssesses 
 an abelian normal subgroup $\mathcal{N}$ such that $\G/\mathcal{N}$ is abelian then necessarily $\mathcal{N}\supset [\G,\G]$ which implies that
 $ \G$ is metabelian. In other words, metabelian groups are characterized by the existence of a normal abelian subgroup $\mathcal{N}$, containing  $[\G,\G]$,
 such that $\G/\mathcal{N}$
 is abelian. We mention that several previous studies of nilpotent Lie groups were done under the additional assumption that the group is metabelian, even if this extra structure was
not explicitly recognized by the authors (e.g. \cite{anzaldo2003goursat,bravo2022geodesics,kruglikov2017integrability,montgomery1997nonintegrable,ardentov2022cut,sachkov2021conjugate,ardentov2013conjugate,bravo2023no} and others).
We refer the reader to \cite[Chapter 5]{robinson2012course} for more details and properties of \ma groups.
 
  Let $\G$ be a metabelian nilpotent group. A central role  in our study 
 is played by the {\bf \textit{maximal abelian normal subgroups}} of $\G$, which will be 
 denoted by
 $\Ag$. It is not difficult to 
 see  that there always exists
 a maximal abelian normal subgroup $\Ag\subset\G$ containing $[\G,\G]$, but it
 may not be unique. Furthermore, there 
 may exist other maximal abelian normal subgroups
 of $\G$ which do not contain $[\G,\G]$.\footnote{the simplest metabelian group that we found for which this holds is the 7-dimensional group  whose Lie algebra is denoted
 by $147D$ in \cite{Cornucopia}.}  Given a maximal 
 abelian normal subgroup $\Ag\subset \G$, one may think of $\G$ as an   extension of $\G/\Ag$ by $\Ag$, and it is not difficult to check that the group's center $\Z(\G) \subset \Ag$.


In general, a Lie algebra $\Lag$ is called {\bf \textit{metabelian}} if the ideal
$[\Lag,\Lag]$ is abelian. 
 Our connectedness assumption on $\G$ implies that the Lie group
$\G$ is \ma if and only 
if its Lie algebra $\Lag$ is metabelian.
 Furthermore, 
if
 $\Ag\subset \G$ is a maximal abelian
normal subgroup, then its Lie algebra
$\Laa$ is a {\bf \textit{maximal abelian ideal}}
of $\Lag$, and our topological assumptions on $\G$ 
imply that the correspondence also goes in the other direction. Namely,  there is a natural correspondence between maximal abelian normal subgroups of $\G$ and
maximal abelian ideals of $\Lag$.
We will systemetically denote maximal abelian ideals of $\Lag$ by $\Laa$ (even if the underlying 
subgroup $\Ag$ is not specified).  As in the group case, 
maximal abelian ideals may not be
unique and may not necessarily contain $[\Lag,\Lag]$.
However, it is not difficult to see that any maximal abelian ideal $\Laa\subset \Lag$ contains the Lie algebra center
 $\Z(\Lag)$.





\subsection{\texorpdfstring{\(\ \Ag\)}{}-simple \ma nilpotent groups}\label{sss:A-sim}
We introduce the following class of groups for
which we will show that Theorem \ref{th:main-new} applies for generic
$\mu\in \Lag^*$.

\begin{definition} \label{def:A-simple}
    Let $\Lag$ be a \ma nilpotent Lie algebra and let $\Laa \subset \Lag$ be a maximal abelian ideal. We say that $\Lag$ is {\bf $\Laa$-\textit{simple}} if there exists a basis $\{X_1,\dots, X_n\}$ of a direct complement $\mathfrak{X}$ of $\Laa$ in $\Lag$  and a set of linearly independent vectors  $Y_1, \dots, Y_n \in \Laa$ satisfying 
    $$0\neq [X_i,Y_i] \in \Z(\Lag) \;\;\mbox{for all}\;\; i=1,\cdots,n. $$ 
 Let $\G$ be a connected and simply connected, \ma nilpotent group, and let 
 $\Ag\subset \G$ be a maximal abelian normal subgroup. We say that $\G$ is {\bf $\Ag$-\textit{simple}} if its Lie algebra $\Lag$ is $\Laa$-simple, where $\Laa$ is the Lie algebra of $\Ag$. 
\end{definition}


 The following theorem  states that, for generic momentum values, the symplectic
 reduction by an $\Ag$-simple Lie group $\G$ is equivalent
 to the abelian symplectic reduction by the maximal abelian normal subgroup $\Ag$.
\begin{theorem}
\label{th:generic}
  Let $\G$ be a metabelian nilpotent
$\Ag$-simple Lie group. There exists  an open 
dense subset $\Lagr\subset \Lag^*$ with the following property. Suppose that $\G$ defines a 
    free and proper Hamiltonian action  on
    the symplectic manifold $(M,\omega)$ with equivariant
    momentum map $J_\G:M\to \Lag^*$.
    If $\mu \in \Lagr \cap J_\G(M)$ then
there is a symplectic diffeomorphism between 
    the symplectic reduced spaces $M\dslash_\mu\G$ and $M\dslash_{i^*(\mu)}\Ag$.
\end{theorem}

 Note that Theorem \ref{th:generic},
together with the discussion of 
Section~\ref{ss:coadjoint}, implies that the generic coadjoint
orbits of an $\Ag$-simple Lie group $\G$ are 
symplectomorphic to $T^*(\G/\Ag)$ (equipped with a magnetic
modification of the canonical 
symplectic form). 

Theorem \ref{th:generic} is an immediate consequence of Proposition
\ref{the:main-1} and the following Lemma.

\begin{lemma}
\label{lemm:injective}
     Let $\G$ be a metabelian nilpotent
$\Ag$-simple group. There exists  an open 
dense subset $\Lagr\subset \Lag^*$ such that,
for all  $\mu \in \Lagr $, we have
\begin{enumerate}
    \item[{\em (a)}] $\G_\mu\subset \Ag$,
    \item[{\em (b)}] For the
    direct complement $\mathfrak{X}$ of $\Laa$ in $\Lag$ given
    in Definition~\ref{def:A-simple}, the mapping $\T_{\mu}:\mathfrak{X} \to (\Laa/\Z(\Lag))^*$ defined by Eq.~\eqref{eq:defTmu}
    is injective.
\end{enumerate}
\end{lemma}

The proof of Lemma
\ref{lemm:injective} relies on Proposition \ref{lem:kirilloc-A-simple-bis} given below
that allows us to work with a convenient basis
of an $\Laa$-simple Lie algebra $\Lag$. In what follows, we will use the following index convention:
\begin{itemize}
    \item The lower case indices $i,j,k$ run from $1$ to $n=\dim \mathfrak{X}$ (as in Definition \ref{def:A-simple}). 
    \item The upper case index $I$ runs from $1$ to $\dim (\Z(\Lag))-1.$ 
    \item The lower case index $a$ runs from $1$ to $\dim \Lag - \dim (\Z(\Lag)) -2n .$ 
\end{itemize}

\begin{proposition}
\label{lem:kirilloc-A-simple-bis}
  Let $\Lag$ be an $\Laa$-simple
  Lie algebra and let $\mathfrak{X}$ be the direct complement
  of $\Laa$ in $\Lag$
  given in Definition \ref{def:A-simple}. There exists a
  basis of $\Lag$ of the form
  \begin{equation*}
      \{Z_0,Z_I,Y_j,Y_a,X_i\},
  \end{equation*}
  such that
   \begin{equation*}
     \Z(\Lag)=\langle  Z_0,Z_I \rangle,  \qquad  \Laa= \langle  Z_0,Z_I, Y_j,Y_a \rangle, \qquad \mathfrak{X}=\langle X_i \rangle ,
  \end{equation*}
  and
  \begin{equation}
  \label{eq:key-comm-rel}
     [X_i,Y_j] =\delta_{ij}Z_0+\sum_IC^I_{ij}Z_I+\sum_aC^a_{ij}Y_a,
  \end{equation}
where $\delta_{ij}$ is Kronecker's delta and    $C_{ij}^I, C^a_{ij}$ are  some structure coefficients.
\end{proposition}

We present a proof of Lemma  \ref{lemm:injective}
and give the proof of Proposition \ref{lem:kirilloc-A-simple-bis}
afterwards.

\begin{proof}[Proof of Lemma  \ref{lemm:injective}] 
Work with the basis of Proposition \ref{lem:kirilloc-A-simple-bis} and  denote by
$$\{Z_0^*,Z_I^*,Y_j^*,Y_a^*,X_i^*\},$$ 
the dual basis
for $\Lag^*$.
Introduce linear coordinates $(c,\epsilon_I,\beta_k, \gamma_a , \alpha_k)$ on $\Lag^*$ by writing $\mu\in \Lag^*$ as
\begin{equation}
    \mu=cZ_0^*+\sum_I\epsilon_I Z_I^*+\sum_k(\alpha_k X_k^*+\beta_kY^*_k)+\sum_a\gamma_a Y_a^*.
\end{equation}

Now, for $\mu \in \Lag^*$, let $\mathcal{M}(\mu)$ be the $n\times n$ matrix with
entries 
 $$\mathcal{M}(\mu)_{ij}:=\langle \ad^*_{-X_i}\mu,Y_j \rangle_\Lag.$$
 It is clear that the entries of  $\mathcal{M}(\mu)$
 depend linearly on the coordinates  $(c,\epsilon_I,\beta_k, \gamma_a , \alpha_k)$.

Now consider $X=\sum_i \alpha_iX_i\in \mathfrak{X}$. Using the definition \eqref{eq:defTmu} of $\T_\mu$, and considering that $Y_j\in \Laa$, we
have
\begin{equation*}
\begin{split}
\langle \T_\mu(X), \overline{Y_j} \rangle_{\Laa/\Z(\Lag)} & = \langle i^*(\ad^*_{-X}\mu),Y_j\rangle_\Laa  \\
& =\langle \ad^*_{-X}\mu,Y_j\rangle_\Lag= \sum_i \alpha_i \mathcal{M}(\mu)_{ij}.
\end{split}
\end{equation*}
This shows that if $X\in \ker \T_\mu$ then the column vector
$(\alpha_1,\dots, \alpha_n)^T\in \R^n$ is a null-vector
of $\mathcal{M}(\mu)^T$. In particular, we conclude that 
$\T_\mu$ is injective whenever the matrix $\mathcal{M}(\mu)$
is invertible.

 Consider the mapping $\psi:\Lag^* \to \R$ given by
    $$\psi(\mu)=\det(\mathcal{M}(\mu)).$$
Then $\psi$ is a  homogeneous
polynomial function of degree $n$ on the coordinates $(c,\epsilon_I,\beta_k, \allowbreak \gamma_a , \alpha_k)$
and from the discussion above, we know that 
 $\T_\mu$ is injective on the open subset $\Lagr$ of
 $\Lag^*$ on which $\psi(\mu)\neq 0$.

Next, using Eq.~\eqref{eq:key-comm-rel}, we find that if  
$\epsilon_I=0$ for all $I$ and $\gamma_a=0$ for all $a$, then

\begin{equation*}
\begin{split}
      \mathcal{M}(\mu)_{ij} & =-\langle \mu,[X_i , Y_j]  \rangle_\Lag \\
    &= - \left \langle cZ_0^*+\sum_k(\alpha_k X_k^*+\beta_kY^*_k) \, ,  \, \delta_{ij}Z_0+\sum_IC^I_{ij}Z_I+\sum_aC^a_{ij}Y_a \right \rangle_\Lag  \\
    &= -c \delta_{ij}.
\end{split}
\end{equation*}
In other words, if  
$\epsilon_I=0$ for all $I$ and $\gamma_a=0$ for all $a$,
then $\mathcal{M}(\mu)$ is a scalar multiple of the
$n\times n$ identity matrix by the factor $-c$. 
In particular 
$$\psi(\mu)=(-1)^nc^n.$$
 This shows
that $\psi$ is not identically zero. Given that
$\psi$ is polynomial, it follows that $\psi$
cannot identically vanish on an open set of
$\Lag^*$. Therefore,  the set $\Lagr\subset \Lag^*$ on which $\psi\neq 0$, apart
from being open, is dense in  $\Lag^*$. This proves item (b)
of Lemma \ref{lemm:injective}.

 To prove item (a)
suppose that $\mu\in\Lagr$, so $\T_\mu$ is injective, and, by contradiction,
suppose that $\Lag_\mu$ is not contained in $\Laa$. 
Given that $\Lag=\Laa\oplus \mathfrak{X}$, there
exists a nontrivial $X\in \mathfrak{X}\cap \Lag_\mu$. Since
$X\in \Lag_\mu$, we have 
$\ad_X^*\mu=0$ which implies  $\T_\mu(X)=0$, contradicting
the injectivity of $\T_\mu$. Hence, we must have $\Lag_\mu\subset \Laa$.  As we mentioned earlier, our topological assumptions
on $\G$ imply that the exponential map $\exp:\Lag \to \G$ is a global diffeomorphism. As a consequence  $\exp(\Lag_{\mu}) = \G_{\mu}$, see e.g. \cite[Lemma 1.3.1]{corwin1990representations}. Exponentiating  the inclusion $\Lag_\mu\subset \Laa$, we conclude that $\G_{\mu}\subset \Ag$ as required.

\end{proof}

 Before proving Proposition~\ref{lem:kirilloc-A-simple-bis}, we establish Propositions~\ref{p:dimZeq1} and \ref{p:dimZgeq2}, which will be needed in the proof.

\begin{proposition}
\label{p:dimZeq1}
 Let $\Lag$ be an $\Laa$-simple
  Lie algebra and let $\mathfrak{X}$ be the direct complement
  of $\Laa$ in $\Lag$
  given in Definition \ref{def:A-simple}. Suppose that $\dim \Z(\Lag)=1$. There exists a
  basis of $\Lag$ of the form
  \begin{equation*}
      \{Z_0,Y_j,Y_a,X_i\},
  \end{equation*}
  such that
   \begin{equation*}
     \Z(\Lag)=\langle  Z_0 \rangle,  \qquad  \Laa= \langle  Z_0, Y_j,Y_a \rangle, \qquad \mathfrak{X}=\langle X_i \rangle ,
  \end{equation*}
   satisfying 
    $$[X_i,Y_j]=\delta_{ij} Z_0.$$
\end{proposition}

\begin{proof}
Let  $ \widetilde{X}_i \in \mathfrak{X}$, and $ \widetilde{Y}_j \in \Laa$ be the vectors in Definition \ref{def:A-simple} 
of $\Laa$-simple, in other words they satisfy $0 \neq [\widetilde{X}_i,\widetilde{Y}_i] \in \Z(\Lag)$. We have $\mathfrak{X}=\langle
\widetilde{X}_i\rangle$ and we define $\mathfrak{D}:= \langle
\widetilde{Y}_j\rangle\subset \Laa$.
If $Z_0 \in \Z(\Lag)$ is a non-zero element, then we may  define a duality pairing $\B:\mathfrak{X}\times\mathfrak{D} \to \R$  by the equation
$$[X,Y] = \B(X,Y)Z_0, \;\text{where}\;X\in \mathfrak{X}\;\text{and}\;Y\in \mathfrak{D}.$$
The condition $[\widetilde{X}_i,\widetilde{Y}_i] \neq 0$ implies that $\B$ is non-degenerate. Using
a standard linear algebra Gram-Schmidt-type procedure it is not
difficult to construct bases   $\{X_i\}$ of
$\mathfrak{X}$ and $\{Y_j\}$ of $\mathfrak{D}$ with the property that $\B(X_i,Y_j) = \delta_{ij}$. The set $\{Y_j\}$ is linearly independent, so we may complete a basis for the ideal $\Laa$ to find the desired basis for $\Lag$.    
\end{proof}

\begin{proposition}
\label{p:dimZgeq2}
    Given an $\Laa$-simple Lie algebra $\Lag$,
there exists a metabelian Lie algebra $\tilde \Lag$ with $\dim (\Z(\tilde \Lag))=1$,  and a surjective Lie algebra homomorphism 
$\pi: \Lag \to \tilde \Lag$ satisfying the following 
properties:
\begin{enumerate}
    \item[{\em (a)}] $\ker \pi \subset \Laa$,   
    \item[{\em (b)}] $\pi (\Z(\Lag))=\Z(\tilde \Lag)$.
\end{enumerate}
Moreover, if we define $ \tilde \Laa := \pi (\Laa)$ then $\tilde \Lag$ is $\tilde \Laa$-simple.
\end{proposition}

\begin{proof}
We will proceed by induction on the dimension of $\Z(\Lag)$. If $\dim\Z(\Lag) = 1$, then the Lie algebra homomorphism is the identity map, and the result follows trivially.

Now assume that the result is true for $\Laa$-simple Lie algebras  whose center has
dimension 
$\ell\geq 1$ and let $\Lag_0$ be an $\Laa_0$-simple Lie algebra with 
$\dim\Z(\Lag_0)=\ell+1$.

Consider the decomposition $\Lag_0:= \mathfrak{X}_0\oplus\Laa_0$, and the vectors $X_i\in \mathfrak{X}_0,Y_i\in \Laa_0$ given by the Definition \ref{def:A-simple} of $\Laa$-simple algebra. Let  $0\neq Z \in \Z(\Lag_0)$ satisfying
    \begin{equation}\label{eq:Z-condition2}
        Z  \in \left ( \bigcup_{i}\langle [X_i,Y_i] \rangle \right )^c\cap \Z(\Lag_0).
    \end{equation}
 Such non-zero $Z$ always exists by our assumption that
 $\dim\Z(\Lag_0)=\ell+1\geq 2$. Consider the ideal $\Laz_0:= \langle Z \rangle \subset \Z(\Lag_0)$ of $\Lag_0$ and   the Lie alegebra $\Lag_1 :=\Lag_0/\Laz_0$. The   canonical projection $\pi_{1}:\Lag_0 \to\Lag_1$ is a surjective Lie
 algebra homomorphism satisfying
  \begin{equation}\label{eq:auxpropZ11}
       \ker \pi_1 =\Laz_0 \subset \Z(\Lag_0)\subset \Laa_0, \qquad \dim (\pi_1(\Z(\Lag_0))=\ell.
    \end{equation}
Moreover, given that $\pi_1$ is a surjective Lie algebra homomorphism, we have \begin{equation}\label{eq:auxpropZ12}
\pi_1(\Z(\Lag_0))\subset \Z(\Lag_1),\end{equation}
and we claim that this implies that $\Lag_1$ is $\Laa_1:=\pi_1(\Laa_0)$-simple. To show this, first note that, since
$\pi_1$ is a surjective Lie algebra homomorphism with $\ker \pi_1\subset \Laa_0$, then $\Lag_1$ is metabelian 
and $\Laa_1\subset \Lag_1$ is a maximal abelian ideal.
Next,  we may write $\Lag_1=\mathfrak{X}_1\oplus \Laa_1$ with
$\mathfrak{X}_1:=\pi_1(\mathfrak{X}_0)$. Consider the vectors
$\pi_1(X_i)\in \mathfrak{X}_1$, $\pi_1(Y_i)\in \Laa_1$. Since $\pi_1$
is a Lie algbra homomorphism we have
$$
\left [ \pi_1(X_i),\pi_1(Y_i) \right ]_{\Lag_1}=\pi_1 \left [ X_i,Y_i \right ]_{\Lag_0}.
$$
By construction, $\left [ X_i,Y_i \right ]_{\Lag_0}\notin \Laz_0 =\ker \pi_1$  so $\left [ \pi_1(X_i),\pi_1(Y_i) \right ]_{\Lag_1}\neq 0$, Moreover, since $\left [ X_i,Y_i \right ]_{\Lag_0}\in \Z(\Lag_0)$, we conclude that 
$\left [ \pi_1(X_i),\pi_1(Y_i) \right ]_{\Lag_1}\in \pi_1(\Z(\Lag_0))\subset\Z(\Lag_1)$ which  proves that $\Lag_1$ is
$\Laa_1$-simple as claimed.

Now, in view of Eq.~\eqref{eq:auxpropZ12},
one of the following possibilities holds:
    \begin{enumerate}
        \item[1.] $\pi_1(\Z(\Lag_0))= \Z(\Lag_1) $.
        \item[2.]$\pi_1(\Z(\Lag_0))\subsetneq \Z(\Lag_1)$.
    \end{enumerate}
 
 In the first case, we  may  apply the induction hypothesis to $\Lag_1$ (since $\dim (\Z(\Lag_1))=\ell$)
 to conclude the
 existence of a Lie algebra $\tilde \Lag$, with $\dim (\Z(\tilde \Lag))=1$ and a surjective
 Lie algebra homomorphism $\tilde \pi: \Lag_1\to \tilde \Lag$
 satisfying items (a), (b)  in the statement of the Lemma
 and such that $\tilde \Lag$ is $\tilde \Laa:=\tilde \pi (\Laa_1)$-simple. We claim that  the proof of the proposition follows by taking $\pi:=\tilde \pi\circ \pi_1:\Lag_0\to \tilde \Lag$.
Indeed, we have
$$
\ker \pi =\ker \pi_1 \oplus \pi_1^{-1}(\ker\tilde \pi). 
$$
 Hence, using that $\ker \pi_1\subset \Laa_0$ (in view of 
Eq.~\eqref{eq:auxpropZ11}) and $\ker \tilde \pi\subset \Laa_1=\pi_1(\Laa_0)$ (by the induction hypothesis and the
definition of $\Laa_1$) we conclude that $\ker \pi \subset \Laa_0$
and item (a) holds. Also, by the induction hypothesis we
have $\Z(\tilde \Lag)=\tilde \pi (\Z(\Lag_1))$ which together
with the condition 1. above implies $\pi(\Z(\Lag_0))=\Z(\tilde \Lag)$
showing that item (b) also holds.

On the other hand, if condition 2. above holds, there
exists a nontrivial subspace $\Laz_1\subset \Z(\Lag_1)$ such that
$$
\Z(\Lag_1)=\pi_1(\Z(\Lag_0))\oplus \Laz_1.
$$
Then $\Laz_1$ is an ideal in $\Lag_1$ and we consider the Lie
algebra $\Lag_2:=\Lag_1/\Laz_1$. The canonical projection 
$\pi_2:\Lag_1 \to \Lag_2$ is a surjective Lie algebra homomorphism,
which, in analogy with Eq.~\eqref{eq:auxpropZ11}, satisfies,
$$
\ker \pi_2=\Laz_1\subset \Z(\Lag_1)\subset \Laa_1, \qquad \dim(\pi_2(\Z(\Lag_1)))=\dim(\pi_1(\Z(\Lag_0)))=\ell.
$$
Moreover, since $\pi_2$ is surjective, we have 
$\pi_2(\Z(\Lag_1))\subset \Z(\Lag_2)$ and, arguing as above,  one can
show that $\Lag_2$ is metabelian and is $\Laa_2:=\pi_2(\Laa_1)$-simple. At this point, we are again faced
with a dichotomy in which one of the following holds:
    \begin{enumerate}
        \item[1.] $\pi_2(\Z(\Lag_1))= \Z(\Lag_2) $.
        \item[2.]$\pi_2(\Z(\Lag_1))\subsetneq \Z(\Lag_2)$.
    \end{enumerate}

In case 1., one may apply the induction hypothesis
to $\Z(\Lag_2)$ (since $\dim (\Z(\Lag_2))=\ell$) and there exists a
  Lie algebra $\tilde \Lag$, with $\dim (\Z(\tilde \Lag))=1$ and a surjective
 Lie algebra homomorphism $\tilde \pi: \Lag_2\to \tilde \Lag$
 satisfying items (a), (b) in the statement of the Lemma
 and such that $\tilde \Lag$ is $\tilde \Laa:=\tilde \pi (\Laa_2)$-simple. In analogy with the above, the proof of the proposition follows by taking $\pi:=\tilde \pi \circ \ \pi_2\circ \pi_1:\Lag_0\to \tilde \Lag$ (see below).

In case 2., we repeat the above construction and consider  a nontrivial subspace $\Laz_2\subset \Z(\Lag_2)$ such that
$$
\Z(\Lag_2)=\pi_2(\Z(\Lag_1))\oplus \Laz_2.
$$
Then $\Laz_2$ is an ideal in $\Lag_2$ and we consider the metabelian Lie
algebra $\Lag_3:=\Lag_2/\Laz_2$ and the canonical projection 
$\pi_3:\Lag_2 \to \Lag_3$ which is a surjective Lie algebra homomorphism satisfying,
$$
\ker \pi_3=\Laz_2\subset \Z(\Lag_2)\subset \Laa_2, \qquad \dim(\pi_3(\Z(\Lag_2)))=\dim(\pi_2(\Z(\Lag_1)))=\ell.
$$
Again, using $\pi_3(\Z(\Lag_2))\subset \Z(\Lag_3)$ and, arguing as above,  one can
show that $\Lag_3$ is $\Laa_3:=\pi_3(\Laa_2)$-simple,  and we may repeat the full argument. 

 Considering that $\Lag_0$ is finite dimensional, the above procedure
 ends at some point and may be summarized as follows. 
 There exists a finite sequence of 
 Lie algebras $\Lag_0, \dots, \Lag_R$, for some $R\geq 1$, and surjective Lie algebra homomorphisms $\{\pi_{r+1}:\Lag_r\to \Lag_{r+1}\}_{r=0}^{R-1}$
 with the property that  $\Lag_r$ is
 $\Laa_r$ simple with  $\Laa_{r+1}:=\pi_{r+1}(\Laa_{r})$ for all $r=0,\dots, R-1$. Furthermore, by construction,
 \begin{equation}
 \label{eq:auxpropZ13}
     \begin{split}
         &\ker \pi_{r+1}\subset  \Laa_r,  \qquad r=0,\dots, R-1,\\ 
         & \Z(\Lag_{r+1})=\pi_{r+1}(\Z(\Lag_r))\oplus \ker \pi_{r+2} , \qquad r=0,\dots, R-2, \\
         &\pi_{R}(\Z(\Lag_{R-1}))= \Z(\Lag_{R}), \\
         &\dim(\Z(\Lag_R)))=\ell.
     \end{split}
 \end{equation}

Now, applying the induction hypothesis to $\Lag_R$, there exists 
$\tilde \pi:\Lag_R\to \tilde \Lag$, surjective Lie algebra
homomorphism, with $\dim \Z(\tilde \Lag)=1$, which satisfies
$\ker\tilde \pi\subset \Laa_R$, $\tilde \pi (\Z(\Lag_R))=\Z(\tilde \Lag)$, and such that $\tilde \Lag$ is $\tilde \Laa:=\tilde \pi(\Laa_R)$-simple. Using Eq.~\eqref{eq:auxpropZ13}, it is not difficult to verify (as we did in the case $R=1$ above)   that
the Lie algebra homomorphism $\pi:=\tilde \pi\circ \pi_R\circ \cdots \circ \pi_1:\Lag_0\to \tilde \Lag$ satisfies the properties (a)
and (b) in the statement of the Lemma and also $\tilde \Laa=\pi (\Laa_0)$.

\end{proof}

We are now ready to give the proof of Proposition \ref{lem:kirilloc-A-simple-bis}.

\begin{proof}[Proof of Proposition \ref{lem:kirilloc-A-simple-bis}] 
If $\dim \Z(\Lag)=1$ the result follows automatically
from Proposition \ref{p:dimZeq1}.

For the case  $\dim \Z(\Lag)>1$, we
apply 
Proposition \ref{p:dimZgeq2} which guarantees the existence
of a Lie algebra $\tilde \Lag$, with $\dim \Z(\tilde \Lag)=1$
and a surjective  Lie algebra homomorphism
$\pi:\Lag\to \tilde \Lag$, satisfying
$\ker \pi\subset \Laa$, $\Z(\tilde \Lag)=\pi (\Z(\Lag))$, and
such that $\tilde \Lag$ is $\tilde \Laa:=\pi (\Laa)$-simple.

 Now, given that  $\dim \Z(\tilde \Lag)=1$, we
 may apply Proposition \ref{p:dimZeq1}
to obtain a basis 
 $\{\tilde Z_0,\tilde Y_j,\allowbreak \tilde Y_{\tilde a},\tilde X_i\}$ 
 of $\tilde \Lag$ such that
\begin{equation*}
     \Z(\tilde \Lag)=\langle  \tilde  Z_0\rangle,  \qquad  \tilde \Laa= \langle  \tilde Z_0, \tilde Y_j, \tilde Y_{\tilde a} \rangle, \qquad \mathfrak{\tilde X}=\langle \tilde X_i\rangle 
  \end{equation*}
  and
  \begin{equation}
  \label{eq:comm-cond-aux}
     [\tilde X_i, \tilde Y_j] =\delta_{ij}\tilde Z_0.
  \end{equation}
  We note that the indices $i,j$ run from $1$ to $n=\dim \mathfrak{X}$ since $\dim \mathfrak{X}=\dim \mathfrak{\tilde X}$
  given that $\ker \pi \subset \Laa$, $\Lag=\Laa \oplus \mathfrak{X}$ and $\tilde \Laa$ was defined as $\pi(\Laa)$. On the other hand,  the index $\tilde a$
 runs from $1$ to 
 \begin{equation}
 \label{eq:range-index-tildea}
     \dim \tilde \Lag - \dim (\Z(\tilde \Lag)) -2n =\dim \Lag -\dim \ker \pi -1-2n .
 \end{equation}

Now, the condition that $\Z(\tilde \Lag)=\pi (\Z(\Lag))$ ensures
the existence of $Z_0\in \Z(\Lag)$ such that $\pi(Z_0)=\tilde Z_0$.
On the other hand,  given that $\pi:\Lag \to \tilde \Lag$ 
is surjective, there exists an injective linear map $\varphi:\tilde \Lag \to \Lag$ such that $\pi\circ \varphi=\mbox{id}_{\tilde \Lag}$.  
Let 
$$Y_j:=\varphi(\tilde Y_j),\qquad Y_{\tilde a}:=\varphi(\tilde Y_{\tilde a}), \qquad X_i:=\varphi(\tilde X_i).$$
Then $\{Z_0,Y_j,Y_{\tilde a},X_i \}$ is a linearly independent subset
in $\Lag$, which satisfies
\begin{equation}
\label{eq:auxbasisconstruction}
    \pi(Z_0)=\tilde Z_0, \qquad \pi(Y_j)=\tilde Y_j,\qquad \pi(Y_{\tilde a})=\tilde Y_{\tilde a}, \qquad  \pi(X_i)=\tilde X_i.
\end{equation}
Moreover, using that $\pi$ is a Lie algebra homomorphism,
and $\tilde Y_j, \tilde Y_{\tilde a},\tilde X_i\notin \Z(\tilde \Lag) $, 
we have 
$$\langle Z_0,Y_j,Y_{\tilde a},X_i \rangle \cap \Z(\Lag)=\langle Z_0\rangle.$$

We construct the desired basis of $\Lag$ by adjoining
a convenient basis of $\ker \pi$. Specifically, considering that
$\ker \pi\subset \Laa$ and $\Z(\Lag)\subset \Laa$, we construct
a basis  $\{Z_I,Y_{\hat a}\}$ of $\ker \pi$ such that
$\ker \pi \cap \Z(\Lag)=\langle Z_I \rangle$. Then
$$
\{Z_0,Z_I,Y_j,Y_{\hat a},Y_{\tilde a}, X_i\}
$$
is a basis of $\Lag$ which we claim has the desired properties. 
First note that the index $I$ runs from $1$ to $\dim (\Z(\Lag))-1$
as it should. Indeed, the condition $\pi(\Z(\Lag ))=\Z(\tilde \Lag)=\langle \tilde Z_0\rangle$ implies that $\dim (\ker \pi \cap \Z(\Lag))= \dim (\Z(\Lag))-1$, as desired. As a consequence,
$\Z(\Lag)=\langle Z_0,Z_I\rangle$. Moreover, 
the index $\hat a$ runs
on the (possibly empty) range from 1 to
\begin{equation*}
  \dim  \ker \pi -( \dim (\Z(\Lag))-1).
\end{equation*}
In view of Eq.~\eqref{eq:range-index-tildea}, this means that the
combined range of the indices $\tilde a$ and $\hat a$ is 
from 1 to $\dim \Lag - \dim (\Z(\Lag))-2n$ which is the
desired range of the index $a$ in the statement of the proposition.

 It only remains to show that the commutation relations
 \eqref{eq:key-comm-rel}
 hold. Using  Eqs.~\eqref{eq:comm-cond-aux}, \eqref{eq:auxbasisconstruction} and the fact
 that 
$\pi$ is a Lie algebra homomorphism, we get
\begin{equation*}
    \begin{split}
    \pi \left ( [X_i,Y_j] \right )
     = [\tilde X_i, \tilde Y_j]   =\delta_{ij}\tilde Z_0=\pi (\delta_{ij}Z_0).
    \end{split}
\end{equation*}
Hence,
\begin{equation*}
    \begin{split}
        [ X_i,  Y_j]  -  \delta_{ij} Z_0 \in \ker \pi =\langle Z_I,Y_{\hat a} \rangle.
    \end{split}
\end{equation*}
Therefore, there exist scalars $ C^I_{ij}, C^{\hat a}_{ij}\in \R$ such that
\begin{equation*}
    \begin{split}
       [ X_i,  Y_j] =   \delta_{ij} Z_0  +       
        \sum_i C^I_{ij}  Z_I +\sum_{\hat a}  C^{\hat a}_{ij} Y_{\hat a},
    \end{split}
\end{equation*}
as required. 
\end{proof}

\subsection{Examples of \texorpdfstring{\(\ \Ag\)}{}-simple groups}\label{sss:A-sim-exa}

 Although Definition \ref{def:A-simple} may appear technical and
difficult to check in practice, we show below that 
many metabelian nilpotent
groups are $\Ag$-simple. We point out that
a necessary condition for $\G$ to be $\Ag$-simple is
that
\begin{equation*}
   \dim \mathfrak{X} = \dim (\Lag/\Laa) \leq \dim (\Laa/\Z(\Lag)).
\end{equation*}
Indeed, this conclusion follows since the vectors $Y_1, \dots, Y_n\in \Laa$ in Definition \ref{def:A-simple} satisfy $Y_i \notin \Z(\Lag)$. The above
inequality can be equivalently written
as 
\begin{equation}
\label{eq:dimCond}
   \dim \Lag+ \dim (\Z(\Lag))\leq 2\dim \Laa.
\end{equation}

 In Table~\ref{Table:examples} of the introduction, we list the low-dimensional Carnot groups from~\cite{Cornucopia}, indicating which are $\Ag$-simple and, among those that are not $\Ag$-simple, which fail to be metabelian. The table shows that, in low dimensions, most metabelian nilpotent groups are $\Ag$-simple.
One can verify that all non-$\Ag$-simple metabelian groups appearing in the table violate condition~\eqref{eq:dimCond}. On the other hand, to find examples of metabelian nilpotent groups that are not $\Ag$-simple yet satisfy Eq.~\eqref{eq:dimCond}, one must consider groups $\G$ of dimension at least 9.


\subsubsection{The Heisenberg group }\label{subsub:Hei-group}
The best-known example of a \ma $\Ag$-simple  group is the Heisenberg group. 

    The Heisenberg group $\G=\mathbb{H}^{2n+1}$ is a Carnot group of step $2$ and dimension $(2n+1)$ whose Lie algebra is given by
    \begin{equation*}
        [X_i , Y_i] = Z, \;\;\text{for}\;\;i=1,\dots,n,
    \end{equation*}
  with all other brackets equal to zero.  In this example $\Z(\Lag)=\langle Z \rangle$ and $\Laa = \langle Z,Y_1,\cdots,Y_n\rangle$ is a maximal abelian ideal. Therefore, the above relations show that $\mathbb{H}^{2n+1}$ satisfies Definition \ref{def:A-simple}  with
 $\mathfrak{X}=\langle X_1,\dots, X_n\rangle$ and is therefore  $\Ag$-simple.

\subsubsection{A Carnot group which is not $\Ag$-simple}

On the other hand, the best-known example of a \ma Carnot group that is not $\Ag$-simple is the Cartan group $\mathbb{F}_{2,3}$
(whose Lie algebra is indicated as $N_{5,2,3}$ in Table \ref{Table:examples}) which 
 is a Carnot group of step $3$ and dimension $5$ whose non-trivial brackets are the following
    $$ [X_1,X_2] = Y,\;\;[X_1,Y] = Z_1,\;\;\text{and}\;\;[X_2,Y] = Z_2. $$
     This group has rank $2$ and may be 
    alternatively defined as the free Carnot group with 2 generators and step 3. In this example, $\Z(\Lag)=\langle Z_1,Z_2\rangle$, $\Laa = \langle Z_1,Z_2,Y\rangle$, and  inequality \eqref{eq:dimCond}
    does not hold. Therefore, $\mathbb{F}_{2,3}$    is not $\Ag$-simple.
    
    In contrast, the free Carnot group of step 4 and rank 2,
    denoted $\mathbb{F}_{2,4}$, is $\Ag$-simple. Indeed, this group
    has dimension $8$ and the following non-trivial brackets
    \begin{equation*}
        \begin{split}
            [X_1,X_2]  & = Y_3,\;\;[X_1,Y_3] = Y_2,\;\;[X_2,Y_3] = Y_1, \\
            [X_1,Y_2] = Z_1,\;\; & [X_1,Y_1]=[X_2,Y_2] = Z_2, \;\text{and}\; [X_2,Y_1]=Z_3.
        \end{split}
    \end{equation*}
    We have $\Z(\Lag)=\langle Z_1,Z_2,Z_3\rangle$,  $\Laa= \langle Z_1,Z_2,Z_3,Y_1,Y_2,Y_3\rangle$ as a maximal
    abelian ideal and the bracket relations $[X_1,Y_1]=[X_2,Y_2] = Z_2$ imply that $\mathbb{F}_{2,4}$ is
    an $\Ag$-simple group (with $\mathfrak{X}=\langle X_1,X_2\rangle$).
    
    This is an example of a common phenomenon: it is often the case 
    that adding an extra step to a non-$\Ag$-simple nilpotent group makes it $\Ag$-simple.

\subsubsection{Metabelian Carnot groups with a one-dimensional center}
Nilpotent Lie algebras with a one-dimensional center are important because they provide a foundational building block for understanding more complex Lie algebras and groups \cite{corwin1990representations,kirillov2025lectures,Kirillov-lemma,LATORRE2023271}. Here we will show that every nilpotent, stratifiable, metabelian  Lie algebra $\Lag$ with a one-dimensional center is $\Laa$-simple, for any maximal abelian ideal $\Laa$
containing $[\Lag,\Lag]$.

\begin{thm}\label{prp:genericA}
Let $\Lag$ be a \ma stratified Lie algebra and
let $\Laa \subset \Lag$ be a maximal abelian 
ideal containing $[\Lag, \Lag]$.   If $\dim \Z(\Lag) = 1$, then $\Lag$ is $\Laa$-simple. Consequently, let $\G$ be a \ma Carnot group 
and let $\Ag\subset \G$ be a maximal abelian normal subgroup containing $[\G,\G]$. If $\Z(\G)$ is one-dimensional, then $\G$ is $\Ag$-simple.
\end{thm}

The proof of the theorem relies on two lemmas given below.  In order to state these results, we  
recall that the {\bf \textit{second center}} of a Lie algebra $\Lag$ and the {\bf \textit{centralizer}} of a subset $\Lah \subset\Lag$ are given by
 \begin{equation*}
 \begin{split}
 \Z_{2}(\Lag) & := \{ W\in \Lag : [\Lag,[\Lag,W]] = 0 \},\\
 \C(\Lah,\Lag)& :=\{W\in \Lag:[\Lah,W] = 0 \}. \\
 \end{split}
 \end{equation*}
  It is easily checked that $\Z_{2}(\Lag)$ is a subalgebra. Moreover, $\Z_2(\Lag)$ is a $2$-step nilpotent ideal of $\Lag$ and $\Z(\Lag) \subset\Z_2(\Lag)$ (see e.g. \cite{bourbaki2007groupes}). 


  An essential tool in representation theory is the famous result known as Kirillov's Lemma \cite[Lemma 1.1.12]{corwin1990representations}. More recently, I. Beltita and D. Beltita~\cite{Kirillov-lemma} made the following generalization.
  


\begin{lemma}[Theorem 3.1, \cite{Kirillov-lemma}]\emph{(Generalization of Kirillov's Lemma for Nilpotent algebras)}\label{lem:kirilloc-A-simple}
  Let $\Lag$ be a nilpotent Lie algebra with step larger than two, and $\dim \Z(\Lag) = 1$.  Let $\mathfrak{D}$ and $\mathfrak{W}$ be linear subspaces of $\Lag$ such that 
  \begin{equation}\label{eq:Kirillov-direct-sum}
  \begin{split}
      \Z_2(\Lag) &= \mathfrak{D} \oplus \Z(\Lag)\quad \text{and}\qquad \Lag = \mathfrak{W} \oplus \C(\Z_2(\Lag),\Lag).\\
  \end{split}
  \end{equation}
   Then,  $ \dim \mathfrak{W}  = \dim \mathfrak{D}$.
  Moreover, for fixed $0\neq Z\in \Z(\Lag)$, if $Y_1,\cdots,Y_n$ is a basis for $\mathfrak{D}$, then there exists a unique basis $X_1,\cdots,X_n$ of $\mathfrak{W}$ such that
  $[X_i,Y_j] = \delta_{ij}Z$.
\end{lemma}


 To apply the above result in our setting, we must first establish the relationship between the second center $\Z_2(\Lag)$
and the maximal abelian ideal $\Laa$ appearing in the statement of Theorem~\ref{prp:genericA}. This is accomplished in the following lemma, under the hypotheses of Theorem~\ref{prp:genericA}; that is, $\Lag$ is metabelian, has a one-dimensional center, is stratifiable, and $\Laa\supset [\Lag,\Lag]$.

  \begin{lemma}\label{lem:step-one-dim-CZ2a}
    Let $\Lag$ be a \ma stratified Lie algebra of step $s>2$ and let  
    $\Laa\subset \Lag$ be a maximal
    abelian ideal containing $[\Lag,\Lag]$.
    If $\dim\Z(\Lag) = 1$, then
    \begin{enumerate}
        \item[{\em (a)}] $\Z_2(\Lag) \subset \Laa$,
        \item[{\em (b)}] $\Laa=\C(\Z_2(\Lag),\Lag)$.
    \end{enumerate}
    \end{lemma}

\begin{proof}
(a) Consider the stratification 
$$\Lag = \V_1 \oplus \dots \oplus \V_s, $$
and note that the condition that $\dim\Z(\Lag) = 1$ implies that $\V_s = \Z(\Lag)$. Indeed, by the grading  $\V_s \subset \Z(\Lag)$ and since $\V_s$ is
non-trivial we must have $\V_s = \Z(\Lag)$.

 Next, we claim that we also have $\Z_2(\Lag)=\V_{s-1}\oplus \V_s$. To see this, note
 that for every stratified Lie algebra of step $s$ one has $\V_{s-1} \oplus \V_{s} \subset \Z_2(\Lag)$. On the other
 hand, if $W\in \Z_2(\Lag)$, then, by definition 
 of $\Z_2(\Lag)$, we have $[\Lag,W]\subset \Z(\Lag)=\V_s$. In 
 particular,  $[\V_1,W]\subset \V_s$ 
 which by the grading implies that $W\in \V_{s-1}\oplus \V_s$.
 Therefore, $\Z_2(\Lag)\subset \V_{s-1}\oplus \V_s$, proving  that $\Z_2(\Lag)=\V_{s-1}\oplus \V_s$ as claimed.

 Finally, for any stratified  Lie algebra we always have 
 \begin{equation}
 \label{eq:auxCenter1strat5}
     [\Lag,\Lag] = \V_2 \oplus \dots \oplus \V_s.
 \end{equation}
 Considering 
 that $s>2$ we conclude that  $\Z_2(\Lag)=\V_{s-1}\oplus \V_s \subset [\Lag,\Lag] $. Combining this with the
 hypothesis that 
 $[\Lag,\Lag]\subset \Laa$ implies that $\Z_2(\Lag)\subset \Laa$
 as required.

(b) We begin by observing that the hypothesis that
$[\Lag,\Lag]\subset \Laa$ together with Eq.~\eqref{eq:auxCenter1strat5} imply 
\begin{equation}
 \label{eq:auxCenter1strat6}
  \V_2 \oplus \dots \oplus \V_s \subset \Laa.
\end{equation}
We now show that there exists a direct complement
$\mathfrak{X}$ of $\Laa$ in $\Lag$ which is contained in $\V_1$.
In other words,  $\mathfrak{X}$ is such that 
\begin{equation}
\label{eq:auxCenter1strat2}
    \Lag=\mathfrak{X}\oplus \Laa, \qquad \mathfrak{X}\subset \V_1.
\end{equation}
For this matter, simply define 
 $\mathfrak{X}$ as any direct
 complement of $\V_1 \cap \Laa$ in $\V_1$. Namely,
 $\mathfrak{X}$ is chosen such that 
 $$
 \V_1=(\V_1\cap \Laa)\oplus \mathfrak{X}.
 $$
It is clear that this definition of $\mathfrak{X}$ satisfies $\mathfrak{X}\cap \Laa=\{0\}$ and $\mathfrak{X}\subset \V_1$ as required.  Finally, we check
that $\Lag=\Laa\oplus \mathfrak{X}$ by using   
Grassmann's formula to compute
\begin{equation*}
    \begin{split}
\dim \mathfrak{X} &= \dim \V_1 - \dim (\V_1 \cap \Laa) \\
&= \dim \V_1 - \left (  \dim \V_1 +\dim \Laa - \dim (\V_1+\Laa)
\right )\\  &=  \dim \Lag -\dim \Laa,
\end{split}
\end{equation*}
where, in the last equality, we have used  
$\V_1 +\Laa =\Lag$ which is an easy consequence of 
Eq.~\eqref{eq:auxCenter1strat6}.

We now prove that our hypotheses on $\Laa$ imply 
\begin{equation}
\label{eq:auxCenter1strat4}
    \C(\Laa,\Lag)=\Laa.
\end{equation}
Given that $\Laa$ is abelian, we obviously have $\Laa\subset\C(\Laa,\Lag)$. On the other hand, 
if $0\neq X\in \mathfrak{X}\subset \V_1$,  we claim that
there exists $Y\in \Laa$
such that $[X,Y]\neq 0$. Indeed, if this were not the
case, then 
$\tilde \Laa :=\Laa\oplus \langle X\rangle $ would be an abelian 
subalgebra. Moreover, since $X\in \V_1$ and $\Lag$ is stratified, it is easy to check, using Eq.~\eqref{eq:auxCenter1strat6}, that $\tilde \Laa$ is an
ideal, contradicting the maximality of $\Laa$. This,
together with the decomposition $\Lag=\mathfrak{X}\oplus \Laa$,
allows us to conclude that if $ W\in \Lag \setminus  \Laa$, there
exists $Y\in \Laa$ such that $[W,Y]\neq 0$, which is equivalent
to the statement $\Laa\supseteq\C(\Laa,\Lag)$ which 
proves Eq.~\eqref{eq:auxCenter1strat4}.

Next, using that $\Z_2(\Lag)\subset \Laa$ (as established in item (a)), we immediately 
obtain
$\C(\Laa,\Lag)\subset \C(\Z_2(\Lag),\Lag)$, which in
view of Eq.~\eqref{eq:auxCenter1strat4} is equivalent to 
\begin{equation}
\label{eq:auxCenter1strat1}
   \Laa\subset \C(\Z_2(\Lag),\Lag),
\end{equation}
so we only need to prove that the opposite inclusion holds.
We will show that 
\begin{equation}
\label{eq:auxCenter1strat3}
   \mathfrak{X}\cap \C(\Z_2(\Lag),\Lag)=\{0\},
\end{equation}
which, together with Eqs.~\eqref{eq:auxCenter1strat2}  and \eqref{eq:auxCenter1strat1},  implies  $\C(\Z_2(\Lag),\Lag)\subset \Laa$ as required. We will make use the identities
$$
\Z(\Lag)=\V_s, \qquad \Z_2(\Lag)=\V_{s-1}\oplus \V_s,
$$
established in the proof of item (a) above.

Proving Eq.~\eqref{eq:auxCenter1strat3} amounts to showing that
for every  $0\neq X\in \mathfrak{X}$, there exists  $Y\in \Z_2(\Lag)$ such that $[X,Y]\neq 0$. Fix then $0\neq X\in \mathfrak{X}$ and let us construct such $Y$.   In view of Eq.~\eqref{eq:auxCenter1strat4}, 
the condition $X\notin \Laa$ implies that there exists  $Y_1\in \Laa$ such that $[X,Y_1] \neq 0$. Suppose  $[X,Y_1]\in \Z(\Lag)=\V_s$.
      Considering that $X\in \V_1$ (by Eq.~\eqref{eq:auxCenter1strat2}), the stratification
      of $\Lag$ implies that $Y_1\in \V_{s-1}\subset \Z_2(\Lag)$, so we may take 
      $Y := Y_1$. 
      If instead, $[X,Y_1] \notin \Z(\Lag)$ then there exists  $X_1\in \Lag$ such that $[X_1,[X,Y_1]] \neq 0$.
Considering that $[X,Y_1]\in \Laa$ and 
$\Laa$ is abelian, we must have $X_1\notin \Laa$ and we
may assume $X_1\in \mathfrak{X}$. By the Jacobi
identity, we have
\begin{equation*}
    \begin{split}
    0\neq [X_1,[X,Y_1]] &= [[X_1,X],Y_1]+[X,[X_1,Y_1]] \\
    & = [X,[X_1,Y_1]],
    \end{split}
\end{equation*}
where we have used that $[[X_1,X],Y_1]=0$ since 
$[X_1,X]\in [\Lag,\Lag]\subset  \Laa$ and $\Laa$ is abelian.
Therefore, the vector $Y_2 := [X_1,Y_1]\in \Laa$ is such that $[X,Y_2] \neq 0$. If $[X,Y_2] \in \Z(\Lag)$, then we argue
as above to conclude that $Y_2\in \Z_2(\Lag)$ and we take 
$Y:=Y_2$. Otherwise, we repeat the construction  to find
$X_2\in \mathfrak{X}$ such that $Y_3:=[X_2,[X_1,Y_1]]\in \Laa$ satisfies
$[X,Y_3]\neq 0$, and so on. In this way, we construct
a list of vectors $X_1,\dots, X_{a-1}\in \mathfrak{X}$ such that
$Y_a:=[X_{a-1},[X_{a-2},\cdots , [X_1,Y_1]\cdots]]\in \Laa$
satisfies $[X, Y_a]\neq 0$. The process must finalize
at some point due to the grading of the algebra. In other
words, there
exists a certain  $r\geq 1$ for which $0\neq [X,Y_r]\in \V_s=\Z(\Lag)$
and, repeating the argument above, we conclude that 
$Y:=Y_r\in \Z_2(\Lag)$.

\end{proof}

We are now ready to give the proof of Theorem \ref{prp:genericA}.

\begin{proof}[Proof of Theorem \ref{prp:genericA} ] 
   Every nilpotent Lie algebra of step $2$ and  $1$-dimensional center is a Heisenberg algebra \cite[Remark 2.4]{Kirillov-lemma}.  We showed that the Heisenberg algebra is $\Laa$-simple in Subsection~\ref{subsub:Hei-group}. Therefore, the statement is true for the step $2$ case. 
    
    If the step $s\geq 3$, we apply Lemmas \ref{lem:kirilloc-A-simple} and  \ref{lem:step-one-dim-CZ2a}.
    Specifically, let $\mathfrak{X}$ be any complement of $\Laa$
    in $\Lag$ and $\mathfrak{D}$ be any subspace such that 
    $\Z_2(\Lag)=\mathfrak{D}\oplus \Z(\Lag)$ (as in the statement of
    Lemma \ref{lem:kirilloc-A-simple}).
    Let $Y_1, \dots, 
    Y_n$ be a basis of $\mathfrak{D}\subset \Z_2(\Lag)$. Considering that $\Laa=\C(\Z_2(\Lag),\Lag)$
    (by Lemma \ref{lem:step-one-dim-CZ2a}), and taking $\mathfrak{W}=\mathfrak{X}$, we conclude from Lemma \ref{lem:kirilloc-A-simple}  the
    existence of a basis  $X_1, \dots, X_n$ of $\mathfrak{X}$ 
    such that $0\neq [X_i,Y_i]\in \Z(\Lag)$. But the linearly independent vectors $Y_1,\dots, Y_n\in \Laa$ since
    $\Z_2(\Lag)\subset \Laa$ by Lemma \ref{lem:step-one-dim-CZ2a}. Hence, the conditions of an $\Laa$-simple algebra in Definition~\ref{def:A-simple} are satisfied.
\end{proof}

\subsubsection{The jet space $\J^k(\R^n,\R^m)$}

It is well-known that the Heisenberg group $\mathbb{H}^{3}$ is diffeomorphic as a Carnot group to the jet-space $\J^1(\R,\R)$ \cite{golo2023jet}. In this section, we will briefly introduce the jet space $\J^k(\R^n,\R^m)$ as a Carnot group and show it is metabelian. We refer the reader to  \cite{warhurst2005jet} for a more extensive explanation of the Carnot structure of the jet spaces. Our goal in this section is to prove that the jet space $\J^k(\R^n,\R^m)$ is $\Ag$-simple
which is the content of the following.

\begin{thm}\label{thm:jet-space}
    The jet-space $\J^k(\R^n,\R^m)$ is an $\Ag$-simple group. 
\end{thm}

 Let us briefly recall the preliminaries to prove the theorem. If $U\subset \R^n$ is an open set and $x_0 \in U$, then we say that two functions $\textbf{f}, \textbf{g} \in C^k(U,\R^m)$  are equivalent at $x_0$, denoted $\textbf{f}\sim_{x_0}\textbf{g}$, if and only if their Taylor expansions of order $k$ at $x_0$ are equal. The {\bf \textit{$k$-jet space}} over $U$ is given by
$$ \J^k(U,\R^m) = \bigcup_{x_0\in U} C^k(U,\R^m)/\sim_{x_0}. $$
We will denote elements in $\J^k(U,\R^m)$ by $j^k_{x_0}(\textbf{f})$.

Let us make this construction in detail for the case $\mathcal{J}^k(\R^n,\R)$ for simplicity.
The number of partial derivatives of order $k$ for a function $f:\R^n\to \R$ is 
$$\begin{pmatrix}
    n+k-1 \\
    k
\end{pmatrix} = \frac{(n+k-1)!}{k!(n-1)!}.$$
If $I$ is a $k$-index, i.e., $I =(i_1,\cdots,i_n)$ satisfies $|I| = i_1=\cdots+i_n = k$, then we will use the following notation
$$ \partial_I f(x_0) = \frac{\partial^k f}{\partial x_1^{i_1} \cdots \partial x_n^{i_n}}(x_0).$$
We denote the set of $k$-indexes by $I(k)$ and let
$$ \overline{I}(k) = I(0)\cup \cdots \cup I(k).$$
For $I\in \overline{I}(k)$ and $t\in \R^n$, we define
$$ I! = i_1! i_2!\cdots i_n!,\;\;\text{and}\;\; t^I = (t_1)^{i_1}(t_2)^{i_2}\cdots (t_n)^{i_n}.$$
The $k$-th oder taylor polynomial of $f$ at $x_0$ is given by
$$ T^k_{x_0}(f)(t) = \sum_{I\in \overline{I}(k)}\partial_I f(x_0) \frac{(t-x_0)^I}{I!}. $$
Therefore, two functions $f\sim_{x_0}g$ if and only if $T^k_{x_0}(f)(t) = T^k_{x_0}(g)(t)$. 

If $\textbf{f} = (f^1,\dots,f^m)$ is a map $f:U \to \R^m$, then we apply the above construction to the coordinates $f^{\ell}:U\to \R$. We can endow $\mathcal{J}^k(\R^n,\R^m)$ with global coordinates as follows, we denote by $ (x,u^{(k)})$ the coordinates of the point $T^k_{x_0}(\textbf{f})(t)$, where 
$$x(j^k_{x_0}(\textbf{f})) = x_0,\;\;\text{and}\;\;u_I^{\ell}(j^k_{x_0}(\textbf{f})) = \partial_{I}f^{\ell}(x_0), \;\text{for}\;I\in \overline{I}(k),\;\; \ell =1,\dots,m.$$
So the formal definition of $u^{(k)}$ is the following
$$ u^{(k)} := \{ u^\ell_I : I \in \overline{I}(k),\;\;\ell = 1,\dots,m \}. $$

The jet space $\J^k(U,\R^m)$ has a natural distribution $\mathcal{D}^k_{x_0}$defined by the following set of Pfaffian  equations 
$$ 0 = du_I^{\ell} - \sum_{i=1}^n u^\ell_{I+e_i} dx^i,\;\text{for all}\;I \in \overline{I}(k-1),\;\text{and}\;\ell=1,\cdots,m.$$
The distribution $\mathcal{D}^k_{x_0}$ has rank $n+m(\frac{(n+k-1)!}{k!(n-1)!})$, and is globally framed by the vector fields
\begin{equation}
\label{eq:XYjet}
    \begin{split}
        X_i & := \frac{\partial}{\partial x_i}+ \sum_{\ell=1}^m\sum_{I \in \overline{I}(k-1)} u_{I+e_i} \frac{\partial}{\partial u_I^{\ell}},\;\;\text{where}\;\; i=1,\cdots,n,\\
        Y_I^\ell & := \frac{\partial}{\partial u_I^\ell},\;\;\text{where}\;\;I\in \overline{I}(k), \;\text{and}\;\ell=1,\cdots,m.
    \end{split}
\end{equation}
The non-trivial commutators are
\begin{equation}\label{eq:Lie-bra-j-spa}
    [ Y^\ell_{I+e_i},X_i] = Y^\ell_{I}, \;I \in \overline{I}(k-1), \;\text{and}\;\ell=1,\cdots,m.  
\end{equation}

Evaluating these vector fields at the origin $(x,u^{(k)})=(0,0)$, we define the Lie algebra $\Lag :=\mathfrak{j}^k(\R^n,\R^m)$ 
having the same commutation relations.
It admits the stratification 
$$ \Lag = V_1\oplus\cdots\oplus V_{k},$$
whose layers are given by
\begin{equation*}
    \begin{split}
        V_1 & = \mbox{span} \left \{ \frac{\partial}{\partial x_1},\cdots,\frac{\partial}{\partial x_m} \right \} \oplus \mbox{span} \left \{ \frac{\partial}{\partial u_I^{\ell}} \right \}_{I\in I(k)} ,\\
        V_j & = \left \{\frac{\partial}{\partial u_I^\ell}: \;I\in I(k-j+1)\;\text{and}\;\ell=1,\cdots,m \right \}, \;\text{where}\;j=2,\cdots,k.
    \end{split}
\end{equation*}
The Lie bracket relations show that $V_{j+1} = [V_1,V_{j}]$, where $j=1,\cdots,k$, and $0 = [V_{i},V_{j}]$ for all $i,j>1$. It follows that $\Lag$ is a $k$-step stratified \ma nilpotent Lie algebra, so $\G:=\mathcal{J}^k(\R^n,\R^m)$ is a \ma Carnot group. Actually, the maximal abelian ideal $\Laa$ and the center $\Z(\Lag)$ are given by
$$ \Laa = \langle Y_{I}^{\ell} \rangle_{I\in \overline{I}(k), \ell=1,\dots,m},\qquad \Z(\Lag) = \langle Y_{I}^{\ell} \rangle_{I\in I(0), \ell=1,\dots,m},  $$
where it is understood that the vector fields $Y_I^\ell$ defined by Eq.~\eqref{eq:XYjet} are evaluated at the origin. 
Via the  Baker–Campbell–Hausdorff formula, the group multiplication of $\mathcal{J}^k(\R^n,\R^m)$ may be determined, see \cite[Section 4]{warhurst2005jet}. 

We now present the proof of Theorem \ref{thm:jet-space}.
\begin{proof}[Proof of Theorem \ref{thm:jet-space}.] If we define  $\mathfrak{X}:=\langle X_i\rangle_{i=1,\dots, n}$  with the 
vector fields  $X_i$ defined by Eq.~\eqref{eq:XYjet}  evaluated at the origin,
then we have $\Lag=\mathfrak{X}\oplus \Laa$ and the commutation
relations \eqref{eq:Lie-bra-j-spa} show that the definition of $\Laa$-simple is
satisfied if we take $Y_j:=Y_{e_j}^1\in \Laa$. Indeed, 
 Eq.~\eqref{eq:Lie-bra-j-spa} implies that 
$$ [X_i,Y_j] = \delta_{ij} Y^1_{0}\in \Z(\Lag), \;\; \text{for}\;\; i,j=1,\cdots,n.$$ 
Therefore, $\J^k(\R^n,\R^m)$ is $\Ag$-simple. 
\end{proof}

\appendix

\section{ Self-contained proofs of Lemma~\ref{lem:StagesMyk} and 
Theorem~\ref{th:main-new}. }
\label{app}

We begin with the following proposition about 
the adjoint of the linear map $\T_\mu$. 

\begin{proposition}
\label{prop:AdTmu}
Suppose $\Z(\Lag)\subset \Laa$. 
The adjoint  map, $\T^*_{\mu}:(\Laa/\Z(\Lag)) \to \mathfrak{X}^*$, of the linear map $\T_\mu$ defined by
Eq.~\eqref{eq:defTmu} is given by
$$ \T^*_{\mu}(\overline{Y}) = \ad_{i(Y)}^*(\mu).$$
\end{proposition}

\begin{proof}
 If $X\in \mathfrak{X}$ and $\overline{Y} \in \Laa/\Z(\Lag)$, then, by definition of $\T_\mu$,
    \begin{equation*}
    \begin{split}
    \langle\T_{\mu}(X),\overline{Y}\rangle_{\Laa/\Z(\Lag)} & = \langle i^*(\ad^*_{-X}(\mu)),Y\rangle_{\Laa} = \langle \ad^*_{-X}(\mu),i(Y)\rangle_{\Lag} \\ 
    & = \langle \mu,\ad_{-X}(i(Y))\rangle_{\Lag} = \langle \mu,\ad_{i(Y)}(X)\rangle_{\Lag} \\
     & = \langle \ad^*_{i(Y)}(\mu),X\rangle_{\Lag}.\\
    \end{split}
    \end{equation*}
    To complete the proof that $\T^*_{\mu}(\overline{Y}) = \ad_{i(Y)}^*(\mu)$, we must verify that $\ad^*_{i(Y)}(\mu) \in \Laa^{\circ}=\mathfrak{X}^*$. Indeed, let $Y_1\in \Laa$ be arbitrary, then
    \begin{equation*}
    \begin{split}
    \langle \ad^*_{i(Y)}(\mu), Y_1 \rangle_{\Lag} & = \langle \mu,[Y, Y_1]\rangle_{\Lag} = 0,\\
    \end{split}
    \end{equation*}
    where we used that $\Laa$ is abelian. 
\end{proof}

We now present the proof of Lemma~\ref{lem:StagesMyk}, which we restate  
as Lemma~\ref{eq:stagesAppendix} to facilitate the reading of the proof of Theorem~\ref{th:main-new} below.

\begin{lemma}[Stages Hypothesis \cite{Mykytyuk2008663}] 
\label{eq:stagesAppendix}
Suppose that the dimension condition 
\eqref{eq:dim-cond} holds for $\mu\in \Lag^*$, and write $\nu=i^*(\mu)$. 
If $\tilde \mu\in \Lag^*$ satisfies $ i^*(\tilde \mu)=\nu$, then there exists  $a\in \Ag$ such that 
$$ \mu = \Ad^*_{a^{-1}}(\tilde \mu).$$ 
\end{lemma}

\begin{proof}
Since  the dimension condition \eqref{eq:dim-cond} holds, then  
$\Z(\Lag)\subset \Laa$
(Proposition \ref{prop:Zsubseta}) so
$\T_\mu$ is well-defined and injective (Proposition \ref{prop:Tmu}). Therefore,
$\T_\mu^*$ is surjective. 


For $Y\in \Laa$, putting  $a = \exp(i(Y))\in \Ag \subset \G$, and using the
the explicit form of $\T_\mu^*$ given by Proposition \ref{prop:AdTmu},
we have    
\begin{equation}\label{eq:hyp-SSH-proof}
\begin{split}
    \Ad^*_{a^{-1}}(\mu) & = e^{\ad^*_{-i(Y)}}(\mu) = \sum_{n=0}^\infty \frac{(-1)^n}{n!} (\ad^*_{i(Y)})^n(\mu) \\
    & = \mu - \T^*_{\mu}(\overline{Y}),
\end{split}
\end{equation}
    where we have used  $(\ad^*_{i(Y)})^n(\mu) = 0 $ if $n\geq 2$. Indeed, if $X \in \mathfrak{X}$ is an arbitrary element, we have $[Y,[Y,X]] =0$ since $\Laa$ is an abelian ideal. So
    $$ \langle (\ad^*_Y)^2(\mu), X \rangle_\Lag = \langle \mu, [Y,[Y,X]] \rangle_\Lag = 0. $$

    Now, the condition that   $ i^*(\mu) = i^*(\tilde \mu)$, implies  $\mu - \tilde \mu \in \ker i^* = \Laa^\circ = \mathfrak{X}^*$. Since $\T_\mu^*$ is surjective, there  exists $Y_0 \in \Laa$ such that $\T_{\mu}^*(\overline{Y_0}) = \mu - \tilde \mu$. Putting $a_0 = \exp(i(Y_0))$ and using Eq.~\eqref{eq:hyp-SSH-proof}, we conclude that
    $$\Ad^*_{a_0^{-1}}(\mu) = \mu - \T_{\mu}^*(\overline{Y_0}) = \tilde \mu, $$
    as required.
  \end{proof}

We are finally ready to present a self-contained proof of Theorem \ref{th:main-new}, which  is an adaptation of the more general reduction-by-stages theory of Marsden et al.~\cite{stages}. We follow closely the presentation of ``Point Reduction by Stages'' in Section 5.2 of \cite{stages}, and attempt to use the same notation throughout. Several parts of the argument simplify in our setting, since we work under more restrictive assumptions (abelianity of $\Ag$). Lemma~\ref{eq:stagesAppendix} will be used at several points.


\begin{proof}[Proof of Theorem \ref{th:main-new}]
Let $\mu\in \Lag^*$ satisfying the hypothesis of the 
theorem and write $\nu:=i^*(\mu)\in \Laa^*$. 
By item (a) of Proposition \ref{prop:trivialstages}, we have 
\begin{equation}
    \label{eq:GmuAappendix}
    \G_\mu\subset \Ag.
\end{equation}

On the other hand, in view
of Eq.~\eqref{eq:JAJG}, we have
$J^{-1}_\G(\mu)\subset J_\Ag^{-1}(\nu)$. Denote by $s_\mu:J^{-1}_\G(\mu)\hookrightarrow J_\Ag^{-1}(\nu)$
the inclusion map. Then $s_\mu$ is a smooth 
embedding. Indeed, $J^{-1}_\G(\mu)$ is the 
zero level set of the smooth map 
$$C_\mu: J_\Ag^{-1}(\nu)\to \Laa^\circ, \qquad C_\mu(m)=J_\G(m)-\mu,$$
which can be checked to be submersive (see \cite{Enzo}).

Denote by
$$
P_\mu:=J_\G^{-1}(\mu)/\G_\mu \quad \mbox{and} \quad P_\nu:=J_\Ag^{-1}(\nu)/\Ag,
$$
the symplectic reduced manifolds, and by
$\pi_\mu:J_\G^{-1}(\mu)\to P_\mu$ and $\pi_\nu:J_\Ag^{-1}(\nu)\to P_\nu$ the orbit projection maps (which, under our
assumptions are surjective 
submersions).

Let $F:P_\mu\to P_\nu$ be given by
\begin{equation}
\label{eq:defF-appendix}
    F(\pi_\mu (m)):=\pi_\nu\circ s_\mu (m), \qquad \mbox{where $m\in J_\G^{-1}(\mu).$}
\end{equation}
We prove below that $F$ is (i) well-defined, (ii) bijective 
and (iii) smooth.

\begin{enumerate}
    \item[(i)] Suppose that $\pi_\mu (m)=\pi_\mu (\tilde m)$ for 
$m, \tilde m\in J_\G^{-1}(\mu).$ Given that $\G_\mu\subset \Ag$, 
there exists $a\in \G_\mu\subset \Ag$ such that $\tilde m=a\cdot m$, which implies
\begin{equation*}
    \pi_\nu\circ s_\mu(m)= \pi_\nu\circ s_\mu(\tilde m).
\end{equation*}
In other words, the value of $F$ does not depend on 
the orbit representative
 of  $\pi_\mu (m)$ and is well-defined.

\item[(ii)] {\em $F$ is injective.} Suppose that $F(\pi_\mu(m))=F(\pi_\mu(\tilde m))$ for $m,\tilde m\in J_\G^{-1}(\mu)$. By definition of $F$, this implies the
existence of $a\in \Ag$ such that $\tilde m = a\cdot m$. In particular, we conclude that\footnote{$\G\cdot m$ denotes the $\G$-orbit through $m$. Similar obvious interpretations apply to    $\G_\mu\cdot m$ and   $\Ag\cdot m$ below.} 
$$\tilde m\in J_\G^{-1}(\mu)\cap (\G\cdot m).$$
Now recall the well-known identity (see e.g. item (ii)
of the ``Reduction Lemma" in \cite{stages}):
\begin{equation}
    \label{eq:RedLemma}
    J_\G^{-1}(\mu)\cap (\G\cdot m)=\G_\mu\cdot m,
\end{equation}
to conclude that $\tilde m=g\cdot m$ for some $g\in \G_\mu$.
Hence $\pi_\mu(m)=\pi_\mu(\tilde m)$ showing that $F$ is injective.

\noindent {\em $F$ is surjective.} Let $m\in J_\Ag^{-1}(\nu)$ and accept for the moment that
\begin{equation}
\label{eq:aux1-appendix}
    (\Ag \cdot m)\cap J_\G^{-1}(\mu)\neq \emptyset.
\end{equation}
Then, there exists a representative $\tilde m\in J_\G^{-1}(\mu)$
of $\pi_\nu(m)\in P_\nu$, and we have,
$$
F(\pi_\mu(\tilde m))=\pi_\nu(\tilde m)=\pi_\nu(m).
$$
 Therefore, the surjectivity of $F$ may be established by showing that 
Eq.~\eqref{eq:aux1-appendix} holds for all $m\in J_\Ag^{-1}(\nu)$. To prove
 this, let $\tilde \mu:=J_\G(m)\in \Lag^*.$ We have,
\begin{equation}
\label{eq:aux2-appendix}
    i^*(\tilde \mu) =i^* \circ J_\G( m)=J_\Ag(m)=\nu=i^*(\mu),
\end{equation}
 where we have used $J_\Ag=i^*\circ J_\G$ (Eq.~\eqref{eq:JAJG}). Therefore, 
 by Lemma \ref{eq:stagesAppendix}, there exists $a\in \Ag$ such  that $\mu =\Ad_{a^{-1}}^*\tilde \mu.$ Let $\hat m:=a\cdot m$. Then obviously, $\hat m\in \Ag\cdot m$.
 On the other hand, by the equivariance of $J_\G$,
 $$
 J_\G(\hat m)=J_\G(a\cdot m)=\Ad_{a^{-1}}^*(J_\G(m))=\Ad_{a^{-1}}^*(\tilde \mu)=\mu,
 $$
 proving that $\hat m\in    (\Ag \cdot m)\cap J_\G^{-1}(\mu)$
and hence the intersection \eqref{eq:aux1-appendix} is non-empty as claimed.
 \item[(iii)] The smoothness of $F$ follows from the smoothness of $\pi_\nu\circ
 s_\mu$ and the fact that $\pi_\mu$ is a surjective submersion.
\end{enumerate}

We now show that $F:P_\mu\to P_\nu$ is a symplectic map, i.e.,
$$
F^*\omega_\nu=\omega_\mu,
$$
where we recall that the symplectic reduced forms $\omega_\mu\in \Lambda^2(P_\mu)$ and
$\omega_\nu\in \Lambda^2(P_\nu)$ are characterized by the conditions
$$
\pi_\mu^*\omega_\mu=j_\mu^*\omega, \qquad \pi_\nu^*\omega_\nu=j_\nu^*\omega,
$$
 where $j_\mu:J_\G^{-1}(\mu)\hookrightarrow M$ and $j_\nu:J_\Ag^{-1}(\nu)\hookrightarrow M$ are the inclusion maps. For this, it is 
 sufficient to show that 
 $$
\pi_\mu ^*F^*\omega_\nu=j_\mu^*\omega.
$$
Using that $F\circ \pi_\mu=\pi_\nu \circ s_\mu$ by definition of $F$, we have,
\begin{equation*}\begin{split}
    \pi_\mu ^*F^*\omega_\nu& =(F\circ \pi_\mu)^* \omega_\nu = (\pi_\nu\circ s_\mu)^*\omega_\nu= s_\mu^*\pi_\nu^* \omega_\nu= s_\mu^*j_\nu^*\omega \\ &= (j_\nu \circ s_\mu)^*\omega =  j_\mu^*\omega,
    \end{split}
\end{equation*}
where, in the last identity, we have used  $j_\mu=j_\nu\circ s_\mu$ 
which follows from the definition of these inclusion maps.

To finalize the proof of the theorem, it remains only to show that $F$ has a smooth inverse,
which requires substantial work.
We first propose what such an inverse, that we denote as $\phi:P_\nu\to P_\mu$, should be. 
Let $\pi_\nu(m)\in P_\nu$ for some $m\in J_\Ag^{-1}(\nu)$
and let $\tilde \mu:=J_\G(m)\in \Lag^*$. Repeating the 
calculation \eqref{eq:aux2-appendix},
shows that $i^*(\tilde \mu)=\nu=i^*(\mu)$ and therefore, by  Lemma~\ref{eq:stagesAppendix}, there exists $a\in \Ag$ such that  $\mu=\Ad^*_{a^{-1}}\tilde \mu$. We define,
\begin{equation}
\label{eq:aux3-appendix}
    \phi(\pi_\nu(m))= \pi_\mu (a\cdot m).
\end{equation}
We first observe that the right hand side of this formula makes sense, since, by equivariance
of $J_\G$, we have $J_\G(a \cdot m)=\Ad^*_{a^{-1}}J_\G(m)=\Ad^*_{a^{-1}}\tilde \mu=\mu$. Hence, the point $a\cdot m$ belongs to $J_\G^{-1}(\mu)$ and we may apply $\pi_\mu$ to it.

 Similarly to what we did for $F$, we show  below that
(i) $\phi$ is well-defined,
(ii) $\phi$ is the inverse of $F$, and (iii) $\phi$ is smooth. 

\begin{enumerate}
    \item[(i)] Suppose that $\pi_\nu( \hat m)=\pi_\nu(m)$ for $\hat m,   m\in J_\Ag^{-1}(\nu)$ and let $a_0\in \Ag$ be such that $\hat m= a_0\cdot m$. We repeat the construction for $\hat m$ that we did for $m$ in the definition of $\phi$. Let $\hat \mu:=J_\G(\hat m)\in \Lag^*$. Again, since $i^*(\hat \mu)=i^*(\mu)$, by  Lemma~\ref{eq:stagesAppendix}, there exists
    $\hat a\in \Ag$ such that $\hat \mu=\Ad_{\hat a^{-1}}^*\mu$. As a consequence, we may write,
\begin{equation}
\label{eq:aux4-appendix}
    \phi(\pi_\nu(\hat m))=\pi_\mu(\hat a\cdot \hat m)=\pi_\mu(\hat aa_0\cdot  m).
\end{equation}
    Proving that $\phi$ is well-defined amounts to  checking that $\phi(\pi_\nu(m))=\phi(\pi_\nu(\hat m))$, which  in view
    of Eqs.~\eqref{eq:aux3-appendix} and \eqref{eq:aux4-appendix}, 
    is equivalent to:
    $$
    \pi_\mu(a\cdot m)=\pi_\mu(\hat aa_0\cdot  m).
    $$
    In other words, we must show the existence of $g\in \G_\mu$ such that
    $a\cdot m=(g\hat aa_0)\cdot  m$. By freeness of the $\G$-action on $M$, this is
    equivalent to showing that $a^{-1}g\hat aa_0$ is the identity element
    of $\G$. Therefore, to complete the proof, we only need to check that
    $g:=aa_0^{-1}\hat a^{-1}$ belongs to $\G_\mu$. But this is true since
    both points $a\cdot m$ and $\hat a a_0\cdot m$ belong to $J_\G^{-1}(\mu)$
    and also to the orbit $\G\cdot m$. As a consequence, they
    belong to the intersection $J_\G^{-1}(\mu)\cap \G\cdot m$ which
    by Eq.~\eqref{eq:RedLemma} equals $\G_\mu\cdot m$. Therefore, we conclude
    that $a$ and $\hat a a_0$ belong to $\G_\mu$, and therefore also $g=a(\hat a a_0)^{-1}\in \G_\mu$.

    \item[(ii)] Let $m\in J_\Ag^{-1}(\nu)$ and let $a\in \Ag$ be the group element
    which defines  $\phi (\pi_\nu(m))$ according to Eq.~\eqref{eq:aux3-appendix}.
    In view of the definition \eqref{eq:defF-appendix} of $F$, we have,
\begin{equation*}
    F\circ \phi(\pi_\nu(m))=F(\pi_\mu(a \cdot m))=\pi_\nu(s_\mu(a\cdot m))
    =\pi_\nu(a \cdot m)=\pi_\nu(m).
\end{equation*}
This shows that  $F=\phi^{-1}$ since we know that $F$ is bijective.

    \item[(iii)] Let $m_0\in J_\Ag^{-1}(\nu)$ and consider a local section $\Gamma:U\to \pi_\nu^{-1}(U)\subset J_\Ag^{-1}(\nu)$ where $U\subset P_\nu$ is
    a neighbourhood of $\pi_\nu(m_0)$. Consider the smooth map
    \begin{equation*}
        J_\G\circ \Gamma : U \to \Lag^*.
    \end{equation*}
    We claim that it takes values on $\Ag\cdot \mu$,
    which
    denotes the $\Ag$-orbit through $\mu$ of the 
    restriction of the coadjoint action of $\G$ on $\Lag^*$ to  $\Ag\subset \G$. Actually,
    we will show that $\left . J_\G \right |_{J_\Ag^{-1}(\nu)}$ does,
    which is enough. Indeed, let $m\in J_\Ag^{-1}(\nu)$, then
    $i^*(J_\G(m))=\nu$ by Eq.~\eqref{eq:JAJG} so, by Lemma~\ref{eq:stagesAppendix},  there exists $a\in \Ag$ such that $\Ad_{a^{-1}}^*(J_\G(m))=\mu$ and our claim holds. 
    
    Now, by Eq.~\eqref{eq:GmuAappendix}, we have $\G_\mu \subset \Ag$ so there is a natural diffeomorphism between the orbit
    $\Ag\cdot \mu$ and the quotient  $\Ag/\G_\mu$. We
    denote such
    diffeomorphism by $\mathcal{C}_\mu : \Ag\cdot \mu \to \Ag/\G_\mu$,
    which is characterised by the condition,
    \begin{equation}
   \label{eq:auxThmASmoothness2}
        \mathcal{C}_\mu (\Ad_{a^{-1}}^*(\mu))=\tau(a), \qquad a\in \Ag,
    \end{equation}
where $\tau: \Ag\to \Ag/\G_\mu$ is the principal bundle projection.

Consider the (smooth) composition 
   \begin{equation*}
       \mathcal{C}_\mu \circ  J_\G\circ \Gamma : U \to \Ag/\G_\mu.
    \end{equation*}
 Restricting the neighbourhood $U$ if necessary, we may consider
   a smooth lift 
\begin{equation*}
       (  \mathcal{C}_\mu \circ J_\G\circ \Gamma)^\ell : U \to \Ag,
    \end{equation*}
which we use to construct the smooth map 
\begin{equation*}
      \alpha: U\to \Ag, \qquad \alpha (\pi_\nu(m))=\left ( ( \mathcal{C} \circ J_\G\circ \Gamma)^\ell(\pi_\nu(m)) \right )^{-1},
    \end{equation*}
where ${}^{-1}$ denotes the group inverse. Finally, we consider the
smooth map
\begin{equation}
   \label{eq:defsigma}
      \sigma : U\to P_\mu, \qquad \sigma (\pi_\nu(m))=\pi_\mu  \left
      ( 
     \alpha (\pi_\nu(m)) \cdot \Gamma (\pi_\nu(m))  \right ).
    \end{equation}

 Our proof that  $\phi$ is smooth consists in showing that 
 $\phi$ coincides with $\sigma$ on $U$, namely 
    \begin{equation}
    \label{eq:eqtoshow}
        \sigma (\pi_\nu(m))= \phi(\pi_\nu(m)), \qquad \mbox{for all
        $\pi_\nu(m)\in U$.}
    \end{equation}

  Fix $\pi_\nu(m)\in U$. Given that $J_\G \circ \Gamma(\pi_\nu(m))\in \Ag\cdot \mu$, there exists $a_0\in \Ag$ such
  that 
\begin{equation*}
  \mu=   \Ad^*_{a_0^{-1}} \left (  J_\G \circ \Gamma(\pi_\nu(m)) \right ).
    \end{equation*}
    On the one hand, in view of the definition \eqref{eq:aux3-appendix} of $\phi$, the above identity implies 
    \begin{equation}
    \label{eq:auxphivalue}
  \phi (\pi_\nu(m))=\pi_\mu (a_0\cdot \Gamma(\pi_\nu(m))).
    \end{equation}

On the other hand, we can write $\Ad^*_{a_0}\mu= J_\G \circ \Gamma (\pi_\nu(m))$, and using the definition of the lift of a map, we get
\begin{equation*}
\begin{split}
    \mathcal{C}_\mu \left ( \Ad^*_{a_0} \mu\right ) &= \mathcal{C}_\mu\circ  J_\G \circ \Gamma \left (\pi_\nu(m)) \right ) \\
     & = \tau  \left (  (\mathcal{C}_\mu\circ  J_\G \circ \Gamma)^\ell (\pi_\nu(m))  \right ).
\end{split}
    \end{equation*}
    By Eq.~\eqref{eq:auxThmASmoothness2}, the above equality implies the existence
    of $g\in \G_\mu$ such that
    \begin{equation*}
        a_0^{-1}g^{-1}=  (\mathcal{C}_\mu\circ  J_\G \circ \Gamma)^\ell (\pi_\nu(m)). 
    \end{equation*}

Therefore, $\alpha (\pi_\nu(m)) = ga_0$. Following the definition \eqref{eq:defsigma} of $\sigma$, and using that $g\in \G_\mu$, 
we have
  \begin{equation*}
       \sigma(\pi_\nu(m)) = \pi_\mu \left ( g a_0 \cdot \Gamma
       (\pi_\nu(m)) \right )= \pi_\mu \left (  a_0 \cdot \Gamma
       (\pi_\nu(m)) \right ).
    \end{equation*}
Comparing this with Eq.~\eqref{eq:auxphivalue} proves Eq.~\eqref{eq:eqtoshow} as required.
\end{enumerate}
\end{proof}

\section*{Acknowledgements}
We are grateful to H. Dullin for
a conversation that led to Remark \ref{rmk:2body}. ABD is thankful to the Department of Mathematics `Tullio Levi-Civita' at the University of Padua for its hospitality
during his visits in 2023 and 2024. During the 2023 visit, a question posed by D. Barilari regarding the dimension of coadjoint orbits during ABD’s talk served as the initial spark for this work. ABD would like to express his sincere gratitude to E. Le Donne for introducing the concept of a metabelian group and to N. Paddeu for the insightful conversations that inspired this work. He also thanks his advisor, R. Montgomery, for providing a fundamental insight at the heart of this work: \enquote{for connected Lie groups, the study of symplectic reduction can be reduced to the Lie algebra.}  We also thank one of the
referees for pointing out reference  \cite{Mykytyuk2008663} on the
stages hypothesis, and T. Ratiu for sharing his personal notes
working out some of the details of \cite{Mykytyuk2008663}.

\section*{Funding}

 LGN acknowledges support from the project MIUR-PRIN 2022FPZEES {\em Stability in Hamiltonian dynamics and beyond}.

\bibliographystyle{plain}
\bibliography{bibli.bib}

\end{document}